\documentclass[reqno,10pt]{amsart}
\usepackage{amsmath}	
\usepackage{amssymb}	
\usepackage{amsthm}
\usepackage{amsfonts}
\usepackage{latexsym}
\usepackage{mathrsfs}
\usepackage[all,ps,cmtip]{xy}
\usepackage{tabularx}
\usepackage{pict2e}
\usepackage{nicematrix}
\usepackage{graphicx}
\usepackage{xypic}
\usepackage{tikz-cd}
\usepackage{tikz}

\usepackage{xcolor}
\usepackage{soul}
\sethlcolor{yellow} 

\usetikzlibrary{matrix,arrows}
\usepackage[colorlinks=true, citecolor=red, urlcolor=blue, linkcolor=blue]{hyperref}

\theoremstyle{plain} 
\newtheorem{thm}{Theorem}[section]
\newtheorem{lem}[thm]{Lemma} 
\newtheorem{prop}[thm]{Proposition} 
\newtheorem{cor}[thm]{Corollary} 

\theoremstyle{definition} 
\newtheorem{defn}[thm]{Definition}
\newtheorem{rem}[thm]{Remark}

\makeatletter
\@namedef{subjclassname@2020}{%
  \textup{2020} Mathematics Subject Classification}
\makeatother

\makeatletter
\def\@tocline#1#2#3#4#5#6#7{\relax
  \ifnum #1>\c@tocdepth 
  \else
    \par \addpenalty\@secpenalty\addvspace{#2}%
    \begingroup \hyphenpenalty\@M
    \@ifempty{#4}{%
      \@tempdima\csname r@tocindent\number#1\endcsname\relax
    }{%
      \@tempdima#4\relax
    }%
    \parindent\z@ \leftskip#3\relax \advance\leftskip\@tempdima\relax
    \rightskip\@pnumwidth plus4em \parfillskip-\@pnumwidth
    #5\leavevmode\hskip-\@tempdima
      \ifcase #1
       \or\or \hskip 1em \or \hskip 2em \else \hskip 3em \fi%
      #6\nobreak\relax
    \dotfill\hbox to\@pnumwidth{\@tocpagenum{#7}}\par
    \nobreak
    \endgroup
  \fi}
\makeatother

\newcommand{\A}{\mathcal{A}}
\newcommand{\T}{\mathcal{T}}
\newcommand{\Q}{\mathbb{Q}}
\newcommand{\ZN}{\mathbb{Z}}
\newcommand{\RN}{\mathbb{R}}
\newcommand{\PB}{\mathbb{P}}
\newcommand{\CO}{\mathcal{O}}
\newcommand{\CN}{\mathbb{C}}
\newcommand{\Ku}{\mathrm{Ku}}
\newcommand{\DC}{\mathrm{D}^{\mathrm{b}}}
\newcommand{\DD}{\mathbb{D}}
\newcommand{\Hom}{\mathrm{Hom}}
\newcommand{\CHom}{\mathcal{H}om}
\newcommand{\RHom}{\mathrm{RHom}}
\newcommand{\CRHom}{R\mathcal{H}om}

\newcommand{\Ext}{\mathrm{Ext}}
\newcommand{\CExt}{\mathcal{E}xt}
\newcommand{\ext}{\mathrm{ext}}
\newcommand{\ch}{\mathrm{ch}}

\newcommand{\Coh}{\mathrm{Coh}}
\newcommand{\Stab}{\mathrm{Stab}}
\newcommand{\CH}{\mathcal{H}}

 \newcommand{\correspondingauthor}[1]{$^{*}$\thanks{* Corresponding author}}

\allowdisplaybreaks

\numberwithin{equation}{section}

\begin{document}

\title{A moduli space of stable sheaves on a cubic threefold}

\author{Shihao Ma \and Song Yang}

\address{Center for Applied Mathematics and KL-AAGDM, Tianjin University, Weijin Road 92, Tianjin 300072, P. R. China}%
\email{shma@tju.edu.cn, syangmath@tju.edu.cn}%

\begin{abstract}
In this paper, we prove that the moduli space $\overline{M}_{X}(v)$ of Gieseker semistable sheaves on a smooth cubic threefold $X$ with Chern character $v=(4,-H,-\frac{5}{6}H^{2},\frac{1}{6}H^{3})$ is non-empty, smooth and irreducible of dimension $8$. 
Moreover, we show that $\overline{M}_{X}(v)$ is isomorphic to a Bridgeland moduli space in the Kuznetsov component of $X$ and provide a set-theoretic description of it.
\end{abstract}

\date{\today}


\maketitle



\section{Introduction}

Moduli spaces of sheaves on algebraic varieties are very interesting and important topics in algebraic geometry. 
They have been extensively studied over the last few decades, and there are many important results concerning moduli spaces of sheaves on surfaces; see \cite{HL10} for an excellent survey.
One of the central problems is as follows:  
{\it classifying the Chern characters of slope or Gieseker semistable sheaves
and determining when there exist stable sheaves with the given Chern character} (see e.g.,  \cite[Problem 4.8]{CHN23} for the surface case).
It is well-known that every moduli space of Gieseker semistable sheaves on a smooth projective variety is a projective scheme. 
However, the non-emptiness, smoothness, and irreducibility of the moduli spaces of sheaves on higher-dimensional varieties are in general much harder to study as mentioned in \cite[\S 5.A]{HL10}.
In fact, the moduli spaces of sheaves on smooth projective threefolds are still not well understood in general;
especially for that of high-rank sheaves.
Even in the case of $\mathbb{P}^{3}$, results are scarce; see e.g.,  \cite{Sch23} for semistable sheaves up to rank $4$.
The study of moduli spaces of sheaves has been revolutionized by the notion of stability conditions on triangulated categories introduced by Bridgeland \cite{Bri07,Bri08}; see e.g.,  \cite{BM23} for an excellent survey.

In this paper, we study the moduli space of Gieseker semistable sheaves of rank $4$ with a fixed Chern character on a smooth cubic threefold.
Let $X$ be a smooth cubic threefold in the complex projective four-space and $H$ the hyperplane section of $X$.
We denote by $\overline{M}_{X}(v)$ the moduli space of $H$-Gieseker semistable sheaves with the Chern character 
$$
v:=(4,-H,-\frac{5}{6}H^{2}, \frac{1}{6}H^{3}).
$$
Note that every smooth hyperplane section of a smooth cubic threefold contains a smooth rational curve of degree $4$ (see e.g.,  \cite{BL15}).
Using a construction of stable vector bundles in \cite{BBF+}, 
we observe that there exist $H$-Gieseker stable vector bundles with Chern character $v$ (see Proposition \ref{non-empty-lem}).
In particular, the moduli space $\overline{M}_{X}(v)$ is non-empty.
The original motivation of this paper is to study the moduli space $\overline{M}_{X}(v)$ by relating it to the corresponding moduli space of Bridgeland stable objects in the {\it Kuznetsov component} $\Ku(X)$ of $X$. 
Let $\DC(X)$ be the bounded derived category  of coherent sheaves on $X$.
Then there is a semiorthogonal decomposition  
$$
\DC(X)=\langle \Ku(X), \CO_{X},\CO_{X}(H) \rangle.
$$
It was shown in \cite{Kuz04} that the Kuznetsov component $\Ku(X)$ is a fractional Calabi--Yau triangulated category of dimension $\frac{5}{3}$ (i.e., the Serre functor $S$ of $\Ku(X)$ satisfying $S^{3}=[5]$).
The numerical Grothendieck group $\mathcal{N}(\Ku(X))=\ZN[I_{\ell}]\oplus \ZN[S(I_{\ell})]$ is a rank-2 lattice generated by $I_{\ell}$ and $S(I_{\ell})$, where $I_{\ell}$ is the ideal sheaf of a line $\ell\subset X$.
Bridgeland stability conditions on $\Ku(X)$ have been constructed by Bernardara--Macr{\`{\i}}--Mehrotra--Stellari in \cite{BMMS12}, and recently by Bayer--Lahoz--Macr{\`{\i}}--Stellari in \cite{BLMS23}.
It is of importance to notice that these stability conditions on $\Ku(X)$ are Serre-invariant (see e.g.,  \cite{PY22,FP23,Zha20,JLLZ24}). 
Let $\sigma$ be a Serre-invariant stability condition on $\Ku(X)$.
For two coprime numbers $a, b\in\ZN$,
we denote by $M_{\sigma}(a[I_{\ell}]+b[S(I_{\ell})])$ the Bridgeland moduli space of $\sigma$-stable objects in $\Ku(X)$
with numerical class $a[I_{\ell}]+b[S(I_{\ell})]$;
the existence follows from \cite{BLM+}.
In \cite{PY22}, every non-empty moduli space $M_{\sigma}(a[I_{\ell}]+b[S(I_{\ell})])$ has been proved to be smooth by Laura Pertusi and the second-named author; in particular, these moduli spaces are projective by \cite{VP21} (see also \cite{BM23}).
Li--Lin--Pertusi--Zhao \cite{LLPZ24} showed that these moduli spaces are always not empty and irreducible. 
In recent years, the moduli spaces of Bridgeland semistable objects in $\Ku(X)$ have been extensively investigated; see e.g.,  \cite{LMS15,PY22,APR22,BBF+,Qin23,LZ22,FP23,FLZ23,FGLZ24,CHL+24}, and see \cite{PS23} for an excellent survey.

Let $\overline{M}_{X}(a,b)$ be the moduli space of $H$-Gieseker semistable sheaves with Chern character $a\ch(I_{\ell})+b\ch(S(I_{\ell}))$, where $\ch(I_{\ell})=(1,0,-\frac{1}{3}H^{2},0)$ and $\ch(S(I_{\ell}))=(2,-H,-\frac{1}{6}H^{2},\frac{1}{6}H^{3})$.
For any Serre-invariant stability condition $\sigma$ on $\Ku(X)$,
Feyzbakhsh--Pertusi \cite{FP23} proved that
the moduli space $\overline{M}_{X}(0,1)$ is isomorphic to $M_{\sigma}([S(I_{\ell})])$, 
and the moduli space $\overline{M}_{X}(1,0)$ is isomorphic to $M_{\sigma}([I_{\ell}])$; 
both Gieseker moduli spaces are isomorphic to the Fano variety of lines on $X$ (see also \cite{PY22}).
Bayer--Beentjes--Feyzbakhsh--Hein--Martinelli--Rezaee--Schmidt \cite{BBF+} showed that the moduli space $\overline{M}_{X}(1,1)$ is smooth and irreducible of dimension $4$, and it is isomorphic to $M_{\sigma}([I_{\ell}]+[S(I_{\ell})])$.
Since the Chern character $v=2\ch(I_{\ell})+\ch(S(I_{\ell}))$,
it is natural to ask whether the moduli space $\overline{M}_{X}(v)$ is isomorphic to $M_{\sigma}(2[I_{\ell}]+[S(I_{\ell})])$. 
The first main result of this paper is stated as follows.

\begin{thm}\label{mainthm-iso}
The moduli space $\overline{M}_{X}(v)$ is non-empty and smooth of dimension $8$. 
Moreover, for any Serre-invariant stability condition $\sigma$ on $\Ku(X)$,
the moduli space $\overline{M}_{X}(v)$ is isomorphic to $M_{\sigma}(2[I_{\ell}]+[S(I_{\ell})])$.
\end{thm}

In particular, it follows from Theorem \ref{mainthm-iso} and \cite[Theorem 1.4]{LLPZ24} that the moduli space $\overline{M}_{X}(v)$ is irreducible.
On the other hand, providing a characterization of stable sheaves with a fixed Chern character is also an important problem. Bayer--Beentjes--Feyzbakhsh--Hein--Martinelli--Rezaee--Schmidt  gave an explicit description of $H$-Gieseker semistable sheaves of rank $3$ with Chern character $\ch(I_{\ell})+\ch(S(I_{\ell}))$(See \cite[Theorem 6.1]{BBF+}). Generally speaking, however, it is quite difficult to give a detailed geometric description of high-rank stable sheaves with a fixed Chern character. The second main result of this paper is to provide the following characterization of $H$-Gieseker semistable sheaves of rank $4$ with Chern character $v=2\ch(I_{\ell})+\ch(S(I_{\ell}))$.

\begin{thm}\label{sh-charact-thm}
Let $X$ be a smooth cubic threefold.
A coherent sheaf $E$ on $X$ with Chern character $\ch(E)=v$ is $H$-Gieseker semistable if and only if it is either 
\begin{enumerate}
\item[(1)]  the vector bundle $E_{C}$ given by the left mutation 
$$
E_{C}:=\mathrm{L}_{\CO_{X}} (\DD(I_{C}(H)))[-1],
$$
where $\DD(I_{C}(H)):=\CRHom(I_{C}(H),\CO_{X})[1]$ is the derived dual of the twisted ideal sheaf $I_{C}$ of a degenerate smooth quartic rational curve $C$ or a line-free non-degenerate locally Cohen–Macaulay curve $C$ of degree $4$, arithmetic genus $p_{a}(C)=0$ and $H^{0}(C,\CO_{C})=\CN$;
\item[(2)] or the non-reflexive sheaf $E_{\ell}$ given by the short exact sequence 
$$
\xymatrix@C=0.5cm{
0 \ar[r]^{} & E_{\ell} \ar[r]^{} & \mathcal{U}_{X} \ar[r]^{} & \CO_{\ell}(-H)\ar[r]^{} & 0,
}
$$
where $\ell\subset X$ is a line and $\mathcal{U}_{X}:=\Omega_{\PB^{4}}(H)|_{X}$ is the restriction of twisted cotangent bundle of $\PB^{4}$.
\end{enumerate}
\end{thm}

This theorem gives a set-theoretic description of $\overline{M}_{X}(v)$ and can also yield the smoothness of the moduli space $\overline{M}_{X}(v)$.
In \cite[Theorem 7.3]{HRS05}, Harris--Roth--Starr proved that the Hilbert scheme $\CH^{4,0}(X)$ of smooth quartic rational curves is irreducible of dimension $8$. 
By \cite[Example 8.8]{LLPZ24}, the moduli space $M_{\sigma}([I_{\ell}]+2[S(I_{\ell})])$ is a birational model of $\CH^{4,0}(X)$.
Our Theorem \ref{sh-charact-thm} may also yield that $\overline{M}_{X}(v)$ is a birational model of $\CH^{4,0}(X)$.

The structure of this paper is organized as follows.
In Section \ref{Prelim}, 
we collect some basic facts on (weak) Bridgeland stability conditions, tilt stability, wall-crossing, and stability conditions on Kuznetsov components of smooth cubic threefolds.
In Section \ref{non-emptiness-Modsp},  
we obtain three non-empty moduli spaces of Gieseker semistable sheaves with Chern character $2\ch(I_{\ell})+\ch(S(I_{\ell}))$, $3\ch(I_{\ell})+\ch(S(I_{\ell}))$, and $4\ch(I_{\ell})+\ch(S(I_{\ell}))$.
In Section \ref{Proof-main-result}, 
we prove Theorem \ref{mainthm-iso}, 
by using Bridgeland stability conditions and wall-crossing techniques.
In Section \ref{Charact-tor-sh}, we give a characterization of torsion sheaves with Chern character $(0,H,\frac{5}{6}H^{2},-\frac{1}{6}H^{3})$.
In Section \ref{Proof-main-2result}, 
we prove Theorem \ref{sh-charact-thm} by using wall-crossing techniques.

\subsection*{Acknowledgements}
We would like to thank Laura Pertusi for sending the preprint \cite{LLPZ24}, 
Zhiyu Liu for useful comments, and Shizhuo Zhang for helpful conversations and comments.
In particular, we thank the referee for the insightful comments, which enabled us to correct several inaccuracies in Theorem \ref{sh-charact-thm}, Proposition \ref{more-vector-bundle} and their proofs, 
and also thank the referee for pointing out the argument in Remark \ref{new-arugument-for-OYD}. 
This work is partially supported by the National Natural Science Foundation of China (No. 12171351).


\section{Preliminaries}\label{Prelim}

This section will review the definitions and basic facts of (weak) Bridgeland stability conditions, tilt stability, wall-and-chamber structure, and stability conditions on the Kuznetsov component of a smooth cubic threefold. 
Our main references are \cite{BMS16,BLMS23}.

\subsection{(Weak) stability conditions}

Let $\mathcal{D}$ be a $\CN$-linear triangulated category and $K(\mathcal{D})$ the Grothendieck group of $\mathcal{D}$. 
Fix a lattice $\Lambda$ of finite rank and a surjective morphism $\upsilon:K(\mathcal{D})\rightarrow \Lambda$.
Let  $\A$ be the heart of a bounded $t$-structure on $\mathcal{D}$.

\begin{defn}
A {\it (weak) stability function} on $\A$ is a group homomorphism (called a {\it central charge}) $Z: \Lambda \rightarrow \CN$ such that for every non-zero $E\in \A$, 
$\Im Z(\upsilon)\geq 0, \; \textrm{and}\; \Im Z(\upsilon(E))= 0 \Rightarrow \Re Z(\upsilon)\, (\leq ) < 0$.
\end{defn}

For simplicity, for an object $E\in \A$, we write $Z(E):=Z(\upsilon)$.
A (weak) stability function $\sigma:=(\A, Z)$ defines a notion of {\it $\mu_{\sigma}$-slope}:
for $E\in \A$,  
$$
\mu_{\sigma}(E)
:=   
\begin{cases}
-\frac{\Re Z(E)}{\Im Z(E)} & \textrm{if}\, \Im Z(E)>0; \\
+\infty &  \textrm{otherwise.}
\end{cases}
$$
A non-trivial object $E\in \mathcal{D}$ is called {\it $\sigma$-(semi)stable} if some shift $E[k]\in \A$ and for every proper subobject $F\subset E[k]$ in $\A$, 
we have $\mu_{\sigma}(F) \, (\leq) < \mu_{\sigma}(E[k]/F)$.

\begin{defn}
A {\it (weak) stability condition} on $\mathcal{D}$ (with respect to $\Lambda$) is a pair $\sigma=(\A, Z)$ consists of the heart $\A$ of a bounded $t$-structure on $\mathcal{D}$ and a (weak) stability function $Z$ on $\A$ such that the following conditions hold:
\begin{enumerate}
\item[(i)] Any object of $\A$ has a Harder--Narasimhan filtration in $\sigma$-semistable ones.

\item[(ii)] There exists a quadratic form $\Delta$ on $\Lambda\otimes \mathbb{R}$ such that $\Delta|_{\ker Z}$ is negative definite, and $\Delta(E)\geq 0$ for all $\sigma$-semistable objects $E\in \A$.
\end{enumerate}
\end{defn}

We denote by $\Stab_{\Lambda}(\mathcal{D})$ the set of stability conditions on the $\CN$-linear triangulated category $\mathcal{D}$.
This set $\Stab_{\Lambda}(\mathcal{D})$, called the stability manifold of $\mathcal{D}$, has the structure of a complex manifold \cite[Theorem 1.2]{Bri07}.
Moreover, there exist two natural group actions on the stability manifold $\Stab_{\Lambda}(\mathcal{D})$:
\begin{itemize}
  \item a right action of the universal covering space $\tilde{\mathrm{GL}}^{+}_{2}(\RN)$ of $\mathrm{GL}^{+}_{2}(\RN)$.
  \item a left action the group of $\CN$-linear exact autoequivalences of $\mathcal{D}$.
\end{itemize}
We refer to \cite[Lemma 8.2]{Bri08} for detailed discussion.


\subsection{Tilt stability}

Given a weak stability condition $\sigma=(\A, Z)$ on a $\CN$-linear triangulated category $\mathcal{D}$.
For every real number $\mu\in \mathbb{R}$, 
there is a torsion pair $(\T_{\sigma}^{\mu}, \mathcal{F}_{\sigma}^{\mu})$ defined as follows:
$$
\T_{\sigma}^{\mu}
:=   \langle 
E \in \A \mid E \,\textrm{is}\, \sigma\textrm{-semistable with}\, \mu_{\sigma}(E)>\mu 
\rangle,
$$
$$
\mathcal{F}_{\sigma}^{\mu}
:=   \langle 
E \in \A \mid E \,\textrm{is}\, \sigma\textrm{-semistable with}\, \mu_{\sigma}(E)\leq\mu
\rangle.
$$
We denote by 
$\A^{\mu}_{\sigma}:=  \langle  \T_{\sigma}^{\mu}, \mathcal{F}_{\sigma}^{\mu}[1]\rangle$ the extension closure of $\T_{\sigma}^{\mu}$ and $\mathcal{F}_{\sigma}^{\mu}$.
Then $\A^{\mu}_{\sigma}$ is the heart of a bounded $t$-structure on $\mathcal{D}$ (see \cite{HRS96}).
In particular, $\A_{\sigma}^{\mu}$ is called the heart obtained by {\it tilting} $\A$ with respect to $\sigma$ at the slope $\mu$.

Now, let $X$ be a smooth projective threefold and $H$ an ample divisor.
For every integer $j\in \{0,1, 2, 3\}$, 
we consider the lattice $\Lambda_{H}^{j}\cong \ZN^{j+1}$ generated by 
$$
(H^{3}\ch_{0}, H^{2}\ch_{1}, \cdots, H^{3-j}\ch_{j})\in \Q^{j+1}
$$
with the map $\nu_{H}^{j}: K(\DC(X))\rightarrow \Lambda_{H}^{j}$.
Then, the pair $\sigma_{H}:= (\Coh(X), Z_{H})$ with 
$
Z_{H}(E):=  -H^{2} \cdot \ch_{1}(E)+ i H^{3} \cdot \ch_{0}(E),
$
is a weak stability condition on $\DC(X)$ with respect to the lattice $\Lambda^{1}_{H}$, which is known as {\it $\mu_{H}$-slope stability}.

Now, fixing a real number $\beta\in \RN$,
we denote by $\Coh^{\beta}(X)$ the heart of a bounded $t$-structure on $\DC(X)$ obtained by tilting the $\mu_{H}$-slope stability at the slope $\mu_{H}=\beta$.
For $E\in \DC(X)$, the twisted Chern character is given by $\ch^{\beta}(E):=e^{-\beta H}\cdot \ch(E)$.
More explicitly, the first three formulae are as follows: 
\begin{eqnarray*}
\ch_{0}^{\beta}(E) & = & \ch_{0}(E) \\
\ch_{1}^{\beta}(E) & = & \ch_{1}(E)-\beta H\cdot \ch_{0}(E) \\
\ch_{2}^{\beta}(E) & = & \ch_{2}(E)-\beta H\cdot \ch_{1}(E)+\frac{\beta^{2}}{2}H^{2} \cdot \ch_{0}(E). 
\end{eqnarray*}

\begin{prop}[{\cite[Proposition 2.12]{BLMS23}}]
There is a continuous family weak stability conditions $\sigma_{\alpha, \beta}:= (\Coh^{\beta}(X), Z_{\alpha, \beta})$ (with respect to $\Lambda_{H}^{2}$) with the quadratic form given by $\Delta_{H}$ and the central charge
$$
Z_{\alpha, \beta}(E)
:=   \frac{1}{2}\alpha^2 H^{3}\cdot \ch_{0}^{\beta}(E)-H \cdot \ch_{2}^{\beta}(E)
+ i H^{2}\cdot \ch_{1}^{\beta}(E), 
$$
where $(\alpha,\beta)\in \RN_{>0}\times \RN$.
\end{prop}

\begin{rem}
The weak stability condition $\sigma_{\alpha, \beta}$ is known as the {\it tilt stability} (see \cite{BMT14}).
Here the {\it $H$-discriminant} of $E\in \DC(X)$ is given by
\begin{eqnarray*}
\Delta_{H}(E): 
& = & (H^{2}\cdot \ch^{\beta}_{1}(E))^{2}-2(H^{3}\cdot\ch^{\beta}_{0}(E))\cdot (H \cdot \ch^{\beta}_{2}(E)) \\ 
& = & (H^{2}\cdot \ch_{1}(E))^{2}-2(H^{3}\cdot\ch_{0}(E))\cdot (H \cdot \ch_{2}(E)).
\end{eqnarray*}
There is a version of the Bogomolov inequality in tilt stability  (see \cite[Corollary 7.3.2]{BMT14} and \cite[Theorem 3.5]{BMS16}):
if $E\in \Coh^{\beta}(X)$ is $\sigma_{\alpha,\beta}$-semistable,
then $\Delta_{H}(E)\geq 0$.
\end{rem}

The following large volume limit plays a significant role in tilt stability (see \cite[Proposition 14.2]{Bri08}, \cite[Proposition 4.8]{BBF+} and \cite[Lemma 2.7]{BMS16}).

\begin{prop}[]\label{limit-tilt-stab-Gie-stab-prop}
Let $E\in \DC(X)$ and $\mu_{H}(E)>\beta$.
Then $E\in \Coh^{\beta}(X)$ and $E$ is $\sigma_{\alpha,\beta}$-(semi)stable for $\alpha\gg 0$
if and only if $E$ is a $2$-$H$-Gieseker (semi)stable sheaf.
In particular, if $\gcd(\ch_{0}(E),\frac{H^{2}\cdot \ch_{1}(E)}{H^{3}})=1$,
then $E\in \Coh^{\beta}(X)$ and $E$ is $\sigma_{\alpha,\beta}$-stable for $\alpha\gg 0$
if and only if $E$ is a $\mu_{H}$-slope stable sheaf.
\end{prop}

\begin{rem}
Let $E\in \Coh(X)$.
If $\gcd(\ch_{0}(E),H^{2}\cdot \ch_{1}(E)/ H^{3})=1$,
then for $E$, $\mu_{H}$-slope stable, $2$-$H$-Gieseker stable, 
$H$-Gieseker stable, $H$-Gieseker semistable, $2$-$H$-Gieseker semistable and $\mu_{H}$-slope semistable are equivalent (see e.g.,  \cite[Section 4]{BBF+}).
\end{rem}

\subsection{Wall-and-chamber structure}

\begin{defn}\label{numerical-wall-def}
Let $v\in K(\DC(X))$.
\begin{enumerate}
  \item[(1)] A {\it numerical wall} for $v$ with respect to for $w\in K(\DC(X))$ in tilt stability is a non-trivial proper subset in the upper half plane
$$
W(v,w):=  \{ (\alpha, \beta)\in \RN_{>0}\times \RN \mid \mu_{\alpha,\beta}(v)=\mu_{\alpha,\beta}(w)\}.
$$ 
  \item[(2)] A {\it chamber} for $v$ is a connected component in the complement of the union of numerical walls in the upper half plane. 
\end{enumerate}
\end{defn}

The numerical walls in tilt stability have well-behaved wall-and-chamber structure ( see \cite{Mac14a} or \cite[Theorem 3.1]{Sch20a}). 

\begin{thm}[]\label{tilt-wall-struct-thm}
Let $v \in K(\DC(X))$ with $\Delta_{H}(v) \geq 0$.
All numerical walls for $v$ in the upper half plane are described as follows:
\begin{enumerate}
\item[(1)] A numerical wall for $v$ is either a semicircle centered at $\beta$-axis or a vertical line parallel to $\alpha$-axis in the upper half plane.
\item[(2)] If $\ch_{0}(v) \neq 0$, then there exists a unique numerical vertical wall for $v$ given by $\beta=\mu_{H}(v)$. 
The remaining numerical walls for $v$ consist of two sets of nested semicircular walls whose apexes lie on the hyperbola $\mu_{\alpha, \beta}(v) = 0$. 
\item[(3)] If $\ch_{0}(v) = 0$ and $H^{2} \cdot \ch_{1}(v) \neq 0$, 
then every numerical wall for $v$ is a semicircle whose top point lies on the vertical line $\beta=\frac{H \cdot \ch_{2}(v)}{H^{2} \cdot \ch_{1}(v)}$ in the upper half plane.
\item[(4)] Any two distinct numerical walls intersect empty.
\end{enumerate}
\end{thm}

\begin{defn}\label{wall-def}
A numerical wall $W$ for $v\in K(\DC(X))$ is called an {\it actual wall} for $v$ if there exists a short exact sequence of $\sigma_{\alpha,\beta}$-semistable objects
\begin{equation}\label{def-wall-sequ}
\xymatrix@C=0.5cm{
0 \ar[r]^{} & A \ar[r]^{} & E \ar[r]^{} & B \ar[r]^{} & 0}
\end{equation}
in $\Coh^{\beta}(X)$ for one $(\alpha,\beta)\in W(E,A)$ such that $v=\ch(E)$ and $\mu_{\alpha,\beta}(A)=\mu_{\alpha,\beta}(E)$ define the numerical wall $W$ (i.e., $W=W(E,A)$). 
\end{defn}

Determining actual walls in tilt stability is the key technique in this paper.

\begin{rem}\label{actual-wall-restriction}
By Definition \ref{wall-def}, 
to determine an actual wall, the following restrictions hold:
\begin{enumerate}
  \item[(1)] $\mu_{\alpha,\beta}(A)=\mu_{\alpha,\beta}(E)=\mu_{\alpha,\beta}(B)$;
  \item[(2)] $\Delta_{H}(A)\geq 0, \Delta_{H}(B)\geq 0$;
  \item[(3)] $\Delta_{H}(A)\leq \Delta_{H}(E)$, $\Delta_{H}(B)\leq \Delta_{H}(E)$.
\end{enumerate} 
Moreover, if an actual wall is given by the sequence \eqref{def-wall-sequ} of tilt semistability,
then 
$$
\Delta_{H}(A)+\Delta_{H}(B)\leq \Delta_{H}(E),
$$
and the equality holds if either $A$ or $B$ is a sheaf supported on points (see \cite[Proposition 12.5]{BMS16}).
\end{rem}

In order to determine the actual walls, we need more restrictions. 
Next, for any tilt semistable object $E$ with $\ch_{0}(E)\neq 0$, we write 
$$
\beta_{-}(E):=\mu_{H}(E)-\sqrt{\frac{\Delta_{H}(E)}{(H^{3}\cdot \ch_{0}(E))^{2}}}
\; \textrm{ and }\;
\beta_{+}(E):=\mu_{H}(E)+\sqrt{\frac{\Delta_{H}(E)}{(H^{3}\cdot \ch_{0}(E))^{2}}},
$$
which are determined by the equation $\mu_{\alpha,\beta}(E)=0$.
Then we have the following properties of destabilizing sequences (see e.g.,  \cite[Proposition 2.7]{Sch23} or \cite[\S 3]{BMSZ17}).

\begin{prop}\label{new-cond-for-wall}
Suppose $\xymatrix@C=0.3cm{
0 \ar[r]^{} & A \ar[r]^{} & E \ar[r]^{} & B \ar[r]^{} & 0}$ is a destabilizing sequence for $E$ along a actual wall $W$.
\begin{enumerate}
  \item[(1)] If $E$ has strictly positive rank, then either $\ch_{0}(A)>0$ and $\mu_{H}(A)\leq \mu_{H}(E)$, or $\ch_{0}(B)>0$ and $\mu_{H}(B)\leq \mu_{H}(E)$.
\item[(2)] If $W$ is a semicircle, $E$ has strictly positive rank, $\ch_{0}(A)>0$ and $\mu_{H}(A)\leq \mu_{H}(E)$, then $\beta_{-}(E)<\beta_{-}(A) \leq \mu_{H}(A)<\mu_{H}(E)$.
\end{enumerate}
\end{prop}

In the case of smooth cubic threefolds, 
the boundary of the open region $R_{\frac{1}{4}}$ in \cite{Li19}
is a curve of segment lines
$$
\{(x, nx-\frac{n^{2}}{2}) \mid \frac{2n-1}{2} \leq x\leq \frac{2n+1}{2}, n\in \ZN\} \subset \RN^{2}.
$$

\begin{prop}[{\cite[Proposition 3.2]{Li19}}]\label{Li-bound}
Let $X$ be a smooth cubic threefold. 
Suppose that $E\in \DC(X)$ is $\sigma_{\alpha,\beta}$-stable with $\ch_{0}(E)\neq 0$ for some $\alpha>0$ and $\beta\in \RN$.
Then the point 
$$
\widetilde{v}_{H}(E):=\Big(\frac{H^{2}\cdot \ch_{1}(E)}{H^{3}\cdot \ch_{0}(E)},\frac{H \cdot \ch_{2}(E)}{H^{3}\cdot \ch_{0}(E)}\Big) \in \RN^{2}
$$ 
is not in the region $R_{\frac{1}{4}}$.
In particular, we have
\begin{enumerate}
  \item[(1)] If $|\mu_{H}(E)|\leq \frac{1}{2}$, then $\frac{H\cdot \ch_{2}(E)}{H^{3}\cdot \ch_{0}(E)}\leq 0$;
  \item[(2)] If $\frac{1}{2}<|\mu_{H}(E)| \leq 1$,
  then $\frac{H\cdot \ch_{2}(E)}{H^{3}\cdot \ch_{0}(E)}\leq |\mu_{H}(E)|-\frac{1}{2}$.
\end{enumerate}
If $\widetilde{v}_{H}(E)$ is on the boundary of $R_{\frac{1}{4}}$, then $\ch_{0}(E)=\pm 1$ or $\pm 2$.
\end{prop}

\subsection{Stability conditions on Kuznetsov component of a cubic threefold}
 
Let $X$ be a smooth cubic threefold.
Then there is a semiorthogonal decomposition 
$$
\DC(X)=\langle \Ku(X),\CO_{X}, \CO_{X}(H) \rangle
$$
of the bounded derived category $\DC(X)$, where $\Ku(X)$ is called the {\it Kuznetsov component of $X$}; see \cite{Kuz04,KPS18} for more detailed discussion.
By \cite[Proposition 2.7]{BMMS12}, the numerical Grothendieck group of $\Ku(X)$ is a lattice of rank $2$
$$
\mathcal{N}(\Ku(X))=\ZN[I_{\ell}]\oplus \ZN[S(I_{\ell})],
$$
where $I_{\ell}$ is the ideal sheaf of a line $\ell\subset X$ and $S$ is the Serre functor of $\Ku(X)$.
Stability conditions on $\Ku(X)$ have been constructed in \cite{BMMS12} and recently in \cite{BLMS23} using a different construction. 
Next let us recall the construction {of} stability conditions on $\Ku(X)$ in \cite{BLMS23}.
Consider the tilted heart $\Coh^{0}_{\alpha,\beta}(X)=\langle \mathcal{F}_{\alpha,\beta}[1], \mathcal{T}_{\alpha,\beta} \rangle $ of tilt stability 
$\sigma_{\alpha,\beta}$ at $\mu_{\alpha,\beta}=0$, where $ \mathcal{F}_{\alpha,\beta}$ (resp. $ \mathcal{T}_{\alpha,\beta}$) is the subcategory of $\Coh^{\beta}(X)$ with $\mu_{\alpha,\beta}^{+}\leq 0$ (resp. $\mu_{\alpha,\beta}^{+}>0$).
By \cite[Proposition 2.15]{BLMS23}, the pair $\sigma^{0}_{\alpha,\beta}=(\Coh^{0}_{\alpha,\beta}(X),-iZ_{\alpha,\beta})$ is a weak stability condition on $\DC(X)$.
Following \cite[Theorem 6.8]{BLMS23}, 
the explicit computations in \cite[Theorem 3.3]{PY22} yield a family of stability conditions $\sigma(\alpha,\beta)$ on $\Ku(X)$, where $(\alpha,\beta)$ lies in the triangular region
$$
\mathcal{V}:=\{(\alpha,\beta)\in \RN_{>0}\times \RN_{<0} 
\mid 
-\frac{1}{2} \leq \beta, \alpha < -\beta;  
-1 < \beta <-\frac{1}{2}, \alpha \leq 1+\beta\}.
$$ 
More precisely, we have the following:

\begin{thm}[{\cite[Theorem 6.8]{BLMS23}, \cite[Theorem 3.3]{PY22}}]\label{thm_U}
For any $(\alpha,\beta)\in  \mathcal{V}$, the pair
$\sigma(\alpha, \beta):=(\A(\alpha, \beta), Z(\alpha, \beta))$
is a Bridgeland stability condition on $\Ku(X)$ 
with respect to $\Lambda_{H, \Ku(X)}^{2}:= \mathrm{Im}(K(\Ku(X)) \rightarrow K(X) \to \Lambda_H^{2}) \cong \ZN^{2}$,
where 
$$
\A(\alpha, \beta):=\Coh^{0}_{\alpha,\beta}(X)\cap \Ku(X)
 \; \textrm{ and }\;  
 Z(\alpha, \beta):=-iZ_{\alpha,\beta}|_{\Ku(X)}.
$$
\end{thm}

The following notion plays an important role in the study of moduli spaces of semistable objects in the Kuznetsov component.
 
\begin{defn}
A stability condition $\sigma$ on $\Ku(X)$ is called {\it Serre-invariant} if there is an element $\tilde{g} \in \tilde{\mathrm{GL}}^{+}_{2}(\RN)$ such that $S \cdot \sigma=\sigma \cdot \tilde{g}$, where $S$ is the Serre functor of $\Ku(X)$.
\end{defn}

Moreover, 
we have the following:

\begin{prop}\label{one-orbit-unique-Serre-SC}
For every $(\alpha,\beta)\in  \mathcal{V}$, the stability condition $\sigma(\alpha, \beta)\in \mathcal{K}:=\sigma(\alpha_0,-\frac{1}{2}) \cdot \tilde{\mathrm{GL}}^{+}_{2}(\RN)$  for $0 < \alpha_0 < \frac{1}{2}$ and every stability condition in $\mathcal{K}$ is a Serre-invariant stability condition.
Moreover, all Serre-invariant stability conditions on $\Ku(X)$ lie in the same $\tilde{\mathrm{GL}}^{+}_{2}(\RN)$-orbit.
\end{prop}

\begin{proof}
See, for example, the proof of \cite[Proposition 3.6, Corollary 5.5]{PY22} and \cite[Corollary 4.3]{FP23}.
\end{proof}


\section{Non-emptiness of moduli spaces}\label{non-emptiness-Modsp}

In this section, we construct three non-empty moduli spaces of Gieseker semistable sheaves on a smooth cubic threefold.
From now on, we always assume that $X$ is a smooth cubic threefold.
Let $Y\subset X$ be a hyperplane section and $D$ an effective Weil divisor on $Y$.
For a non-trivial subspace $V\subset H^{0}(X,\CO_{Y}(D))$,
we denote by $\mathcal{E}_{D,V}\in \DC(X)$ the cone of the induced morphism $V\otimes \CO_{X}\rightarrow \CO_{Y}(D)$.
Thus, there is a long exact sequence of cohomology sheaves
\begin{equation}\label{(BBF+)-sh-constrct}
\xymatrix@C=0.5cm{
0 \ar[r]^{} & \CH^{-1}(\mathcal{E}_{D,V}) 
\ar[r]^{} & V \otimes \CO_{X}
 \ar[r]^{} & \CO_{Y}(D)
  \ar[r]^{} & \CH^{0}(\mathcal{E}_{D,V}) 
 \ar[r]^{} & 0.}
\end{equation}
For simplicity, 
we write $E_{D,V}:=  \CH^{-1}(\mathcal{E}_{D,V})$.
If $V=H^{0}(Y,\CO_{Y}(D))$, we set $\mathcal{E}_{D}:=  \mathcal{E}_{D,V}$ and $E_{D}:=  E_{D,V}$.
Now we recall the following construction of stable sheaves in \cite[Lemma 5.1]{BBF+}.

\begin{lem}\label{(BBF+)-sh-constrct-lem}
The sheaf $E_{D,V}$ is $\mu_{H}$-slope stable and reflexive.
If moreover, the sheaf $\CH^{0}(\mathcal{E}_{D,V}) =0$, then $E_{D,V}$ is a locally free sheaf.
\end{lem}

It is known that every smooth hyperplane section contains a smooth rational curve of degree $4$, $5$ and $6$ (see e.g.,  \cite[Lemma 2.8]{BL15}).
We observe the following:

\begin{prop}\label{non-empty-lem}
Let $Y\subset X$ be a smooth hyperplane section and $D\subset Y$ a smooth rational curve of degree $d\in \{4,5,6\}$ as in \cite[Lemma 2.8]{BL15}.
\begin{enumerate}
  \item[(1)] If $d=4$, the $\mu_{H}$-slope stable sheaf $E_{D}$ is a rank-$4$ vector bundle with Chern character 
$\ch(E_{D})=2\ch(I_{\ell})+\ch(S(I_{\ell}))=(4,-H,-\frac{5}{6}H^{2},\frac{1}{6}H^{3})$.
  \item[(2)] If $d=5$, the $\mu_{H}$-slope stable sheaf $E_{D}$ is a rank-$5$ vector bundle with Chern character 
$\ch(E_{D})=3\ch(I_{\ell})+\ch(S(I_{\ell}))=(5,-H,-\frac{7}{6}H^{2},\frac{1}{6}H^{3})$.
  \item[(3)] If $d=6$, the $\mu_{H}$-slope stable sheaf $E_{D}$ is a rank-$6$ vector bundle with Chern character 
$\ch(E_{D})=4\ch(I_{\ell})+\ch(S(I_{\ell}))=(6,-H,-\frac{3}{2}H^{2},\frac{1}{6}H^{3})$.
\end{enumerate}
In particular, the moduli space $\overline{M}_{X}(\ch(E_{D}))$ of $H$-Gieseker semistable sheaves with Chern character $\ch(E_{D})$ is non-empty.
\end{prop}

\begin{proof}
Consider the structure sheaf exact sequence
$$
\xymatrix@C=0.5cm{
0 \ar[r]^{} & \CO_{Y}(-D)
\ar[r]^{} & \CO_{Y}
 \ar[r]^{} & \CO_{D}
 \ar[r]^{} & 0.}
$$
Twisted it by $\CO_{Y}(D)$, we get a short exact sequence
\begin{equation}\label{quartic-curve-SES}
\xymatrix@C=0.5cm{
0 \ar[r]^{} & \CO_{Y}
\ar[r]^{} & \CO_{Y}(D)
 \ar[r]^{} & \CO_{D}(D)
 \ar[r]^{} & 0.}
\end{equation}
Since $K_{Y}=-H$ and $D\subset Y$ is a smooth rational curve of degree $d$ (i.e., $D.H=d$), 
by adjunction formula, we have
$$
-2=D(D+K_{Y})=D^{2}-D.H=D^{2}-d
$$
and thus $D^{2}=d-2$.
It follows that $H^{0}(Y,\CO_{D}(D))\cong H^{0}(\PB^{1},\CO_{\PB^{1}}(d-2)) \cong \CN^{d-1}$.
Consider the long exact sequence of the cohomology of \eqref{quartic-curve-SES}, 
since $H^{1}(Y,\CO_{Y})=0$,  we get $H^{0}(X,\CO_{Y}(D))\cong \CN^{d}$. 
Note that $\CO_{Y}(D)$ is globally generated. 
As a result, it follows from \eqref{(BBF+)-sh-constrct} that there is a short exact sequence
\begin{equation}\label{vector-bundle-ED}
\xymatrix@C=0.5cm{
0 \ar[r]^{} & E_{D}
\ar[r]^{} & \CO_{X}^{\oplus d}
 \ar[r]^{} & \CO_{Y}(D)
 \ar[r]^{} & 0.}
\end{equation}
By Lemma \ref{(BBF+)-sh-constrct-lem}, 
$E_{D}$ is a $\mu_{H}$-slope stable vector bundle of rank $d\in \{4,5,6\}$.

Next, we compute the Chern characters of $\CO_{Y}(D)$ and $E_{D}$.
We may assume $\ch(\CO_{D}(D))=(0,0,\frac{d}{3}H^{2},\ch_{3}(\CO_{D}(D)))$.
Since $\chi(\CO_{D}(D))=d-1$, by the Hirzebruch--Riemann--Roch theorem, we have $\ch_{3}(\CO_{D}(D))=-\frac{1}{3}H^{3}$.
Since $Y$ is a hyperplane section, 
it follows that $\ch(\CO_{Y})=(0,H,-\frac{1}{2}H^{2},\frac{1}{6}H^{3})$.
By \eqref{quartic-curve-SES}, we get
$$
\ch(\CO_{Y}(D))=\ch(\CO_{Y})+\ch(\CO_{D}(D))=(0,H,\frac{2d-3}{6}H^{2},-\frac{1}{6}H^{3}).
$$
Therefore, it follows from \eqref{vector-bundle-ED} that the Chern character of $E_{D}$ is 
$$
\ch(E_{D})=d\ch(\CO_{X})-\ch(\CO_{Y}(D))=(d,-H,-\frac{2d-3}{6}H^{2},\frac{1}{6}H^{3}).
$$
This finishes the proof.
\end{proof}

As a direct consequence, the corresponding Bridgeland moduli spaces are non-empty.

\begin{prop}\label{non-empty-BSCmoduli}
Let $\sigma$ be an arbitrary Serre-invariant stability condition on $\Ku(X)$.
Then, for $d\in \{4,5,6\}$, the moduli space $M_{\sigma}((d-2)[I_{\ell}]+[S(I_{\ell})])$ is non-empty. 
In particular, it is smooth and projective of dimension $8$, $14$ and $22$, for $d=4,5,6$ respectively.
\end{prop}

\begin{proof}
Let $Y\subset X$ be a smooth hyperplane section.
Suppose that $D$ is a smooth rational curve of degree $d\in \{4,5,6\}$ as in Proposition \ref{non-empty-lem}. 
Note that for $i\in \ZN$, $H^{i}(Y,\CO_{D}(D))\cong H^{i}(\PB^{1},\CO_{\PB^{1}}(d-2))$ and
$H^{i}(Y,\CO_{D}(D-H))\cong H^{i}(\PB^{1},\CO_{\PB^{1}}(-2))$.
Since $E_{D}$ is $\mu_{H}$-slope stable, 
so $\Hom(\CO_{X},E_{D})=0=\Hom(\CO_{X}(H),E_{D})$.
By \eqref{quartic-curve-SES}, 
we have $\mathrm{RHom}(\CO_{X},\CO_{Y}(D))=\CN^{\oplus d}[0]$.
Note that $\mathrm{RHom}(\CO_{X},\CO_{X}^{\oplus d})=\CN^{\oplus d}[0]$.
It follows from \eqref{vector-bundle-ED}  that $H^{i}(X,E_{D})=0$ for all $i\in \ZN$.
Moreover, note that $\mathrm{RHom}(\CO_{X},\CO_{D}(D-H))=\CN[-1]$.
Since $Y\in |H|$ is a hyperplane section, we have $\mathrm{RHom}(\CO_{X},\CO_{Y}(-H))=\CN[-2]$.
It follows from \eqref{vector-bundle-ED} tensoring with $\CO_{X}(-H)$ that $H^{i}(X,E_{D}(-H))=0$ for all $i\in \ZN$.
This means that $E_{D}\in \Ku(X)$. 

By Proposition \ref{non-empty-lem},  
$E_{D}$ is $\mu_{H}$-slope stable and $\mu_{H}(E_{D})=-\frac{1}{d}>-\frac{1}{d-1}$.
Thanks to Proposition \ref{limit-tilt-stab-Gie-stab-prop}, $E_{D}\in \Coh^{-\frac{1}{d-1}}(X)$ and it is $\mu_{\alpha,-\frac{1}{d-1}}$-stable for $\alpha \gg 0$.
Note that $H^{2}\cdot \ch^{-\frac{1}{d-1}}_{1}(E_{D})=\frac{1}{d-1}H^{3}$.
This means that any destabilizing subobject $A\subset E_{D}$ along 
the vertical line $\beta=-\frac{1}{d-1}$ must have $H^{2}\cdot \ch^{-\frac{1}{d-1}}_{1}(A)=0$ or $H^{2}\cdot \ch^{-\frac{1}{d-1}}_{1}(A)=\frac{1}{d-1}H^{3}$.
Therefore, either $A$ or the quotient $E_{D}/A$ has infinite tilt-slope, a contradiction.
As a result, $E_{D}$ is $\mu_{\alpha,-\frac{1}{d-1}}$-stable for all $\alpha>0$.
Since $\sigma^{0}_{\alpha,-\frac{1}{d-1}}$ is a rotation of $\sigma_{\alpha,-\frac{1}{d-1}}$,
so $E_{D}$ is $\sigma^{0}_{\alpha,-\frac{1}{d-1}}$-stable.
Since $E_{D}\in \Ku(X)$, it follows that $E_{D}$ is $\sigma(\alpha,-\frac{1}{d-1})$-stable.
As a consequence, by Proposition \ref{one-orbit-unique-Serre-SC}, 
$E_{D}$ is $\sigma$-stable.
It follows that the moduli space $M_{\sigma}((d-2)[I_{\ell}]+[S(I_{\ell})])$ is non-empty.
By \cite[Theorem 1.2]{PY22}, the moduli space $M_{\sigma}((d-2)[I_{\ell}]+[S(I_{\ell})])$ is smooth of dimension $d^{2}-3d+4$ ($d=4,5,6$).
It follows from \cite[Corollary 3.4]{VP21} that it is projective (see also \cite[Theorem 2.4]{BM23}).
\end{proof}

\begin{rem} 
For any Serre-invariant stability condition $\sigma$ on $\Ku(X)$,
the Serre functor $S$ of $\Ku(X)$ induces the following isomorphisms of moduli spaces of $\sigma$-stable objects in $\Ku(X)$:
\begin{equation*} 
\xymatrix@C=0.5cm{
      && M_{\sigma}(2[I_{\ell}]+[S(I_{\ell})]) \ar[ld]^{S}_{\simeq}  &&     \\
   & M_{\sigma}(3[S(I_{\ell})]-[I_{\ell}]) \ar[rr]_{S[1]}^{\simeq}  &&  M_{\sigma}(3[I_{\ell}]-2[S(I_{\ell})]) \ar[ul]^{S}_{\simeq}  }
\end{equation*}
\begin{equation*} 
\xymatrix@C=0.5cm{
      && M_{\sigma}(3[I_{\ell}]+[S(I_{\ell})]) \ar[ld]^{S}_{\simeq}  &&     \\
   & M_{\sigma}(4[S(I_{\ell})]-[I_{\ell}]) \ar[rr]_{S[1]}^{\simeq}  &&  M_{\sigma}(4[I_{\ell}]-3[S(I_{\ell})]) \ar[ul]^{S}_{\simeq}  }
\end{equation*}
\begin{equation*} 
\xymatrix@C=0.5cm{
      && M_{\sigma}(4[I_{\ell}]+[S(I_{\ell})]) \ar[ld]^{S}_{\simeq}  &&     \\
   & M_{\sigma}(5[S(I_{\ell})]-[I_{\ell}]) \ar[rr]_{S[1]}^{\simeq}  &&  M_{\sigma}(5[I_{\ell}]-4[S(I_{\ell})]) \ar[ul]^{S}_{\simeq}  }
\end{equation*}
\end{rem}

\begin{rem}
In general, for any coprime integers $a$ and $b$, 
the non-emptiness of the moduli space $M_{\sigma}(a[I_{\ell}]+b[S(I_{\ell})])$ was obtained in \cite[Theorem 1.1]{LLPZ24}.
\end{rem}


\section{Proof of Theorem \ref{mainthm-iso}}\label{Proof-main-result}
 
This section is devoted to the proof of Theorem \ref{mainthm-iso}.
For convenience, we denote by 
$$
W(\mathrm{c},\mathrm{r})
:=
\{(\alpha,\beta) \in \RN_{>0}\times \RN \mid (\beta-\mathrm{c})^{2}+\alpha^{2}=\mathrm{r}^{2}\}
$$
the semicircle in the $(\alpha,\beta)$-plane, where $\mathrm{c} \in \RN$ and $\mathrm{r}\in \RN_{>0}$.

\subsection{Preparatory lemmas}

Let us begin with determining the walls of tilt semistable objects with Chern character $v=(4,-H,-\frac{5}{6}H^{2},\frac{1}{6}H^{3})$.

\begin{lem}\label{mainwall1-comput}
There are no walls to the left of the vertical wall $\beta=-\frac{1}{4}$ for tilt semistable objects $E$ 
with Chern character $\ch_{\leq 2}(E)=(4,-H,-\frac{5}{6}H^{2})$.
\end{lem}

\begin{proof}
Suppose that there is such an actual wall along $\beta<-\frac{1}{4}$ for $E$ induced by the short exact sequence
$$
\xymatrix@C=0.5cm{
0 \ar[r] & A \ar[r] & E \ar[r] & B\ar[r] & 0
}
$$
of $\sigma_{\alpha,\beta}$-semistable objects in $\Coh^{\beta}(X)$.
We set $\ch_{\leq 2}(A):=(a,bH,\frac{c}{6}H^{2})$ with $a,b,c\in \ZN$.
Since $\ch_{0}(E)=4>0$, based on Proposition \ref{new-cond-for-wall}, 
by changing the roles of $A$ and $B$, 
we may assume that $\ch_{0}(A)>0$ and $\mu_{H}(A)<  \mu_{H}(E)$ (i.e., $a>0$ and $\frac{b}{a}< -\frac{1}{4}$).
Since $0\leq \Delta_{H}(A)\leq \Delta_{H}(E)$, we get
\begin{equation}\label{mainwall1-delta}
3b^{2}-23\leq ac\leq 3b^{2}. 
\end{equation}
Note that $\beta_{-}(E)=-\frac{1}{4}(1+\frac{\sqrt{69}}{3})>-1$ and all numerical walls to the left of the vertical wall intersect with the vertical line $\beta=\beta_{-}(E)$.
Along $\beta=\beta_{-}(E)$, since $A, E \in \Coh^{\beta_{-}(E)}(X)$, we get
$$
0<H^{2}\cdot \ch^{\beta_{-}(E)}_{1}(A)=(b-\beta_{-}(E) a) H^{3} < H^{2}\cdot \ch^{\beta_{-}(E)}_{1}(E).
$$
This gives the following inequality
\begin{equation}\label{A-mu-slope-ineq}
-\frac{1}{4}(1+\frac{\sqrt{69}}{3})<\mu_{H}(A)<\frac{\sqrt{69}}{3a}-\frac{1}{4}(1+\frac{\sqrt{69}}{3}).
\end{equation}
By Proposition \ref{new-cond-for-wall}, we have $-1<\beta_{-}(E)<\beta_{-}(A)\leq \mu_{H}(A)<\mu_{H}(E)$.
If $a\geq 6$, it follows  
$$
-1<\beta_{-}(A) < \beta_{+}(A)=2\mu_{H}(A)-\beta_{-}(A)< \frac{2\sqrt{69}}{3a}-\frac{1}{4}(1+\frac{\sqrt{69}}{3})<0.
$$
Moreover, since $-1<\mu_{H}(A)<-\frac{1}{4}$,
then the point $\widetilde{v}_{H}(A)$
is in the open region $R_{\frac{1}{4}}$.
This is a contradiction with Proposition \ref{Li-bound}.
Therefore, we get $a \in \{ 1,2,3,4,5\}$.
We will rule out all the possibilities case by case.

{\bf Case 1:}  
If $a=5$, by \eqref{A-mu-slope-ineq}, it follows that $b\in \{-2,-3,-4\}$.
(i)  If $b=-4$, then $c$ must be even.
By Proposition \ref{Li-bound} and \eqref{mainwall1-delta}, we get $c\in \{6, 8\}$.
In both cases, the numerical wall $W(E,A)$ is empty.         
 (ii)  If $b=-3$, then $c$ must be odd.
By Proposition \ref{Li-bound} and \eqref{mainwall1-delta},  
we have $c=1$.
Then the numerical wall $W(E,A)$ is  empty. 
(iii) If $b=-2$,  then $c$ must be even.  
By Proposition \ref{Li-bound} and \eqref{mainwall1-delta}, we get $c=-2$ and thus $\ch_{\leq 2}(A)=(5,-2H,-\frac{1}{3}H^{2})$.
By direct computations, 
we have the numerical wall $W(E,A)=W(-\frac{17}{18},\frac{1}{18})$.
By Lemma \ref{auxiliary-lem1-1} below, $A$ has no walls along $\beta=-1$.
Since the vertical line $\beta=-1$ tangents to $W(E,A)$ and $\beta_{-}(A)>-1$,
so $A$ is $\sigma_{\alpha,\beta_{-}(E)}$-semistable for $\alpha\gg 0$.
By Proposition \ref{limit-tilt-stab-Gie-stab-prop}, we obtain that $A$ is a $\mu_{H}$-slope stable sheaf. 
This yields a contradiction with Lemma \ref{auxiliary-lem2} below.

{\bf Case 2:}  
If $a=4$, it follows from \eqref{A-mu-slope-ineq} that $b\in \{-2,-3\}$.
(i)  If $b=-2$, then $c$ is even.
            By Proposition \ref{Li-bound} and \eqref{mainwall1-delta}, we obtain $c= -2$. 
            Then the numerical wall $W(E,A)$ is empty.  
(ii) If $b=-3$, then $c$ is odd.
By Proposition \ref{Li-bound} and \eqref{mainwall1-delta}, we have $c\in \{1,3,5\}$.
Then all the numerical walls $W(E,A)$ are empty.

{\bf Case 3:}   
If $a=3$,  by \eqref{A-mu-slope-ineq},  it follows that $b\in \{-2,-1\}$.
(1) If $b=-2$, then $c$ must be even.
By Proposition \ref{Li-bound} and \eqref{mainwall1-delta}, we get $c\in \{-2,0,2\}$
Then all the numerical walls $W(E,A)$ are empty.  
(2) If $b=-1$, then $c$ must be odd. 
 By Proposition \ref{Li-bound} and \eqref{mainwall1-delta},  
 we have $c\in \{-5,-3,-1\}$.
 \begin{enumerate}
\item[(i)] If $c=-1$, then $\ch_{\leq 2}(A)=(3,-H,-\frac{1}{6}H^{2})$. 
Since $H^{2}\cdot\ch^{-\frac{1}{2}}_{1}(A)=\frac{1}{2}H^{3}$, 
hence $A$ has no walls along $\beta=-\frac{1}{2}$.
Note that the numerical wall $W(E,A)=W(-\frac{11}{6},\frac{\sqrt{73}}{6})$ intersects with $\beta=-\frac{1}{2}$. 
It follows that $A$ is $\sigma_{\alpha,-\frac{1}{2}}$-semistable for $\alpha\gg 0$.
By Proposition \ref{limit-tilt-stab-Gie-stab-prop}, $A$ is a $\mu_{H}$-slope stable sheaf with Chern character $\ch_{\leq 2}(A)=(3,-H,-\frac{1}{6}H^{2})$.
This is a contradiction with \cite[Proposition 6.6]{BBF+}.
\item[(ii)] If $c=-3$,  the numerical wall $W(E,A)$ is empty. If $c=-5$,  the numerical wall $W(E,A)$ is on the right side of the vertical wall.
\end{enumerate}

{\bf Case 4:}   
If $a=2$, by \eqref{A-mu-slope-ineq} and $\mu_{H}(A)<\mu_{H}(E)$, we have $b=-1$ and thus $c$ is odd. 
By Proposition \ref{Li-bound} and \eqref{mainwall1-delta},
we get $-5\leq c\leq -1$ and $c$ is odd.
If $c=-1$, then the numerical wall $W(E,A)$ is empty.
If $-5\leq c\leq -3$, then the numerical wall $W(E,A)$ is on the right side of the vertical wall.

{\bf Case 5:}  
If $a=1$, by \eqref{A-mu-slope-ineq}, we get $b=0,1$, which contradicts $\mu_{H}(A)<\mu_{H}(E)$.
This completes the proof of Lemma \ref{mainwall1-comput}.
\end{proof}

In the next, we will prove Lemma \ref{auxiliary-lem1-1} and Lemma \ref{auxiliary-lem2},
which are used in the proof of Lemma \ref{mainwall1-comput}.

\begin{lem}\label{auxiliary-lem1-1}
There are no walls along the vertical line $\beta=-1$ for tilt semistable objects $E$ with Chern character $\ch_{\leq 2}(E)=(5,-2H,\ch_{2}(E))$, where $\ch_{2}(E)=-\frac{1}{3}H^{2}$ or $\ch_{2}(E)=-\frac{2}{3}H^{2}$.
\end{lem}

\begin{proof}
Suppose that there is such a wall induced by a short exact sequence 
$$
\xymatrix@C=0.5cm{
0 \ar[r] & A \ar[r] & E \ar[r] & B\ar[r] & 0
}
$$
of $\sigma_{\alpha,-1}$-semistable objects in $\Coh^{-1}(X)$.
We write $\ch^{-1}_{\leq 2}(A):=(a,bH,\frac{c}{6}H^{2})$, where $a,b,c \in \ZN$.
Note that $\ch^{-1}_{\leq 1}(E)=(5,3H)$.
By $0< H^{2}\cdot \ch^{-1}_{1}(A)=bH^{3}<H^{2}\cdot \ch^{-1}_{1}(E)=3H^{3}$, 
we get $b=1$ or $2$.
By changing the roles of $A$ and $B$, we may assume $a\geq 3$.
Note that $\ch_{1}(A)=(b-a)H$ and $\ch_{2}(A)=\frac{c-6(b-a)-3a}{6}H^{2}$.
It follows that 
$c_{2}(A)=\frac{\ch_{1}(A)^{2}}{2}-\ch_{2}(A)=\frac{3(b-a)^{2}+6(b-a)+3a-c}{6}H^{2}\in \frac{\ZN}{3}H^{2}$.

{\bf Case 1: $\ch_{2}(E)=-\frac{1}{3}H^{2}$}.
According to $0\leq \Delta_{H}(A)\leq \Delta_{H}(E)$ and $\mu_{\alpha,-1}(A)=\mu_{\alpha,-1}(E)$, 
we obtain $3b^{2}-22\leq ac\leq 3b^{2}$ and $(9a-15b)\alpha^{2}=3c-b$.
By straightforward computations, 
we obtain the following two possible cases:
(i) If $b=1$, then $a=3$ and $c=1$.
Thus, we have $\ch_{\leq 2}(A)=(3,-2H,\frac{2}{3}H^{2})$.
(ii) Suppose $b=2$, there are five possible cases as follows:
\begin{enumerate}
\item[(1)] If $a=3$, then $c=-2$.
Hence, we obtain $\ch_{\leq 2}(A)=(3,-H,-\frac{5}{6}H^{2})$ and $\ch_{\leq 2}(B)=(2,-H,\frac{1}{2}H^{2})$;
\item[(2)] If $a=3$, then $c=0$.
Thus, we get $\ch_{\leq 2}(A)=(3,-H,-\frac{1}{2}H^{2})$ and $\ch_{\leq 2}(B)=(2,-H,\frac{1}{6}H^{2})$;
\item[(3)] If $a=4$, then $c=2$ and thus $\ch_{\leq 2}(A)=(4,-2H,\frac{1}{3}H^{2})$;
\item[(4)] If $a=5$, then $c=2$ and thus $\ch_{\leq 2}(A)=(5,-3H,\frac{5}{6}H^{2})$;
\item[(5)] If $a=6$, then $c=2$ and thus $\ch_{\leq 2}(A)=(6,-4H,\frac{4}{3}H^{2})$.
\end{enumerate}

{\bf Case 2: $\ch_{2}(E)=-\frac{2}{3}H^{2}$}.
By $0\leq \Delta_{H}(A)\leq \Delta_{H}(E)$ and $\mu_{\alpha,-1}(A)=\mu_{\alpha,-1}(E)$, we obtain $3b^{2}-32\leq ac \leq 3b^{2}$ and $3(3a-5b)\alpha^{2}=3c+b$. By direct computations, we obtain the following three possible cases:
(i) If $b=1$, then $a=3$ and $c=1$. Thus, we derive $\ch_{\leq 2}(A)=(3,-2H,\frac{2}{3}H^{2})$.
(ii) Suppose $b=2$ and $a=3$. Then there are three possible cases:
(1) If $c=-2$, then $\ch_{\leq 2}(A)=(3,-H,-\frac{5}{6}H^{2})$ and thus $\ch_{\leq 2}(B)=(2,-H,\frac{1}{6}H^{2})$;
(2) If $c=-4$, then $\ch_{\leq 2}(A)=(3,-H,-\frac{7}{6}H^{2})$ and thus $\ch_{\leq 2}(B)=(2,-H,\frac{1}{2}H^{2})$;
(3) If {$c=-6$}, then $\ch_{\leq 2}(A)=(3,-H,-\frac{3}{2}H^{2})$ and thus $\ch_{\leq 2}(B)=(2,-H,\frac{5}{6}H^{2})$.
(iii) Suppose $b=2$ and $a\geq 4$. Then there are two possible cases:
(1) If $c=0$, then $\ch_{\leq 2}(A)=(a,(2-a)H,\frac{a-4}{2}H^{2})$;
(2)) If $c=2$, then $\ch_{\leq 2}(A)=(a,(2-a)H,\frac{3a-10}{6}H^{2})$.

In both {\bf Case 1} and {\bf Case 2}, all of $A$ or $B$ are ruled out by Proposition \ref{Li-bound}.
This completes the proof of Lemma \ref{auxiliary-lem1-1}.
\end{proof}

To prove Lemma \ref{auxiliary-lem2}, we also need the following:

\begin{lem}\label{auxiliary-lem1-2}
There are no walls along the vertical lines $\beta=0$ for tilt semistable objects $E$ with Chern character $\ch_{\leq 2}(E)=(-5,2H,\frac{1}{3}H^{2})$. 
\end{lem}

\begin{proof}
Assume that there exists such a wall induced by a short exact sequence 
$$
\xymatrix@C=0.5cm{
0 \ar[r] & A \ar[r] & E \ar[r] & B\ar[r] & 0
}
$$
of $\sigma_{\alpha,0}$-semistable objects in $\Coh^{0}(X)$.
Let $\ch_{\leq 2}(A)=(a,bH,\frac{c}{6}H^{2})$, where $a,b,c \in \ZN$.
Then we have $b=1$. 
By changing the roles of $A$ and $B$, we may assume $a\leq -3$.
Using $0\leq \Delta_{H}(A)\leq \Delta_{H}(E)$, we obtain $-19\leq ac\leq 3$.
By $\mu_{\alpha,0}(A)=\mu_{\alpha,0}(E)$, we get $(6a+15)\alpha^{2}=2(c-1)<0$ and thus $c\leq 0$.
Therefore, we have $0\leq ac\leq 3$.
Since $b=1$, hence $c$ is odd.
Thus, we derive $c=-1$ and $a=-3$. 
Then $\ch_{\leq 2}(A)=(-3,H,-\frac{1}{6}H^{2})$, 
a contradiction with Proposition \ref{Li-bound}.
\end{proof}

Based on Lemma \ref{auxiliary-lem1-1} and Lemma \ref{auxiliary-lem1-2}, 
we obtain the following control of the second and third Chern characters. 

\begin{lem}\label{auxiliary-lem2}
Take a $\mu_{H}$-slope stable sheaf $E$ with Chern character $\ch_{\leq 1}(E)=(5,-2H)$.
Then $H\cdot \ch_{2}(E)\leq -\frac{2}{3}H^{3}$.
If moreover, $H\cdot \ch_{2}(E)= -\frac{2}{3}H^{3}$, then $\ch_{3}(E)\leq \frac{1}{3}H^{3}$.
\end{lem}

\begin{proof}
Since $E$ is a $\mu_{H}$-slope stable sheaf, 
by Proposition \ref{Li-bound}, we have $H\cdot \ch_{2}(E)<0$.
Set $\ch_{2}(E):=\frac{c}{6}H^{2}$.
Since $c_{2}(E)=2H^{2}-\ch_{2}(E)\in \frac{\ZN}{3}H^{2}$,
thus $c$ is even and $c\leq -2$.  
It follows that $H\cdot \ch_{2}(E) \leq -\frac{1}{3}H^{3}$.

In the next, we will rule out the case $H\cdot \ch_{2}(E)=-\frac{1}{3}H^{3}$.
We conclude by contradiction and assume that $H\cdot \ch_{2}(E)=-\frac{1}{3}H^{3}$.
We may assume that $E$ is a reflexive sheaf.
Otherwise, one replaces it by the double dual $E^{\vee\vee}$ which is reflexive satisfying 
$H\cdot \ch_{2}(E)\leq H\cdot \ch_{2}(E^{\vee \vee})$. 
Therefore, by the previous argument, we get $H\cdot \ch_{2}(E^{\vee\vee})=-\frac{1}{3}H^{3}$.
Since $E$ is $\mu_{H}$-slope stable and $\mu_{H}(E)=-\frac{2}{5}>-1$,
by Proposition \ref{limit-tilt-stab-Gie-stab-prop},
$E\in \Coh^{-1}(X)$ is $\sigma_{\alpha,-1}$-stable for $\alpha\gg0$.
It follows from Lemma \ref{auxiliary-lem1-1} that $E$ is $\sigma_{\alpha,-1}$-stable for all $\alpha>0$.

On the other hand, 
by hypothesis, $E$ is a $\mu_{H}$-slope stable reflexive sheaf.
Using \cite[Proposition 4.18]{BBF+},  
we get $E[1]$ is $\sigma_{\alpha,-\frac{2}{5}}$-stable for $\alpha>0$. Thus, it is $\mu_{\alpha,0}$-stable for $\alpha\gg0$.
Thanks to Lemma \ref{auxiliary-lem1-2},  
we obtain that $E[1]$ is $\sigma_{\alpha,0}$-stable for $\alpha>0$ 
and thus $E(-2H)[1]$ is $\sigma_{\alpha,-2}$-stable for $\alpha>0$.
By elementary computations, 
the numerical wall $W(E,E(-2H)[1])=W(-\frac{105}{75},\sqrt{\frac{3975}{75}})$
intersects both vertical lines $\beta=-2$ and $\beta=-1$.
Consequently, $E$ and $E(-2H)[1]$ are $\sigma_{\alpha,\beta}$-stable 
for any $(\alpha,\beta)\in W(E,E(-2H)[1])$ and have the same phase.
Therefore, we get $\Hom(E,E(-2H)[1])=0$.
By Serre duality, we obtain 
$$
\Ext^{2}(E,E)\cong \Hom(E,E(-2H)[1])=0.
$$
Since $\Hom(E,E)=\CN$, 
by the Hirzebruch--Riemann--Roch theorem, 
we have 
$$
3=\chi(E,E)=1-\ext^{1}(E,E)-\ext^{3}(E,E)\leq 1,
$$
a contradiction.
This means that $H\cdot \ch_{2}(E) \leq -\frac{2}{3}H^{3}$.

Finally, suppose $H\cdot \ch_{2}(E)= -\frac{2}{3}H^{3}$. 
Since $E$ is a $\mu_{H}$-slope stable sheaf with slope $-2<\mu_{H}(E)=-\frac{2}{5}<0$,
then $\Hom(\CO_{X},E)=$0, and it follows from Serre duality that 
$$
H^{3}(X,E)\cong \Hom(E,\CO_{X}(-2H))=0.
$$
Since $\mu_{H}(E)>-1$, by Proposition \ref{limit-tilt-stab-Gie-stab-prop}, $E$ is $\sigma_{\alpha,-1}$-stable for $\alpha\gg 0$.
According to Lemma \ref{auxiliary-lem1-1}, 
there are no walls for $E$ along the vertical line $\beta=-1$.
Note that $\CO_{X}(-2H)[1]$ is $\sigma_{\alpha,-1}$-stable for $\alpha>0$ and limit tilt-slopes
$$
\lim_{\alpha\to 0^{+}} \mu_{\alpha,-1} (\CO_{X}(-2H)[1])=-\frac{1}{2}<-\frac{1}{18}=\lim_{\alpha\to 0^{+}} \mu_{\alpha,-1}(E).
$$
This yields $\Hom(E,\CO_{X}(-2H)[1])=0$.
It follows from Serre duality that 
\begin{equation*}
H^{2}(X,E) \cong \Hom(E,\CO_{X}(-2H)[1])=0.
\end{equation*}
By the Hirzebruch--Riemann--Roch theorem, we have $\chi(E)=\ch_{3}(E)-\frac{1}{3}H^{3}\leq 0$. 
Therefore, we conclude $\ch_{3}(E)\leq \frac{1}{3}H^{3}$.
\end{proof}

\begin{rem}\label{case-rank-5}
As a direct consequence of Lemma \ref{auxiliary-lem2},
every $\mu_{H}$-slope stable sheaf with Chern character $\ch(I_{\ell})+2\ch(S(I_{\ell}))=(5,-2H,-\frac{2}{3}H^{2},\frac{1}{3}H^{3})$ is reflexive. 
However, we do not know one single $\mu_{H}$-slope stable sheaf with Chern character $\ch(I_{\ell})+2\ch(S(I_{\ell}))$.
A natural question arises:  
{\it is there a $H$-Gieseker stable sheaves with the Chern character 
$\ch(I_{\ell})+2\ch(S(I_{\ell}))$}? 
\end{rem}

\subsection{Proof of Theorem \ref{mainthm-iso}} 

To begin with, we provide the following control of the third Chern character.
 
\begin{prop}\label{mainsh-ch3-char}
Take a $\mu_{H}$-slope stable sheaf $E$ of Chern character $\ch_{\leq 1}(E)=(4,-H)$.
If $H\cdot \ch_{2}(E)= -\frac{5}{6}H^{3}$, then $\ch_{3}(E)\leq \frac{1}{6}H^{3}$.
\end{prop}

\begin{proof}
Since $E$ is a $\mu_{H}$-slope stable sheaf  with slope $-2<\mu_{H}(E)=-\frac{1}{4}$,
then $\Hom(\CO_{X},E)=$0, and it follows from Serre duality that $H^{3}(X,E)\cong \Hom(E,\CO_{X}(-2H))=0$.
Since $\mu_{H}(E)>-1$, by Proposition \ref{limit-tilt-stab-Gie-stab-prop}, $E$ is $\sigma_{\alpha,-1}$-stable for $\alpha\gg 0$.
By Lemma \ref{mainwall1-comput}, 
there are no walls for $E$ along the vertical line $\beta=-1$,
hence $E$ is $\sigma_{\alpha,-1}$-stable for $\alpha>0$.
Note that $\CO_{X}(-2H)[1]$ is $\sigma_{\alpha,-1}$-stable for $\alpha>0$ and limit tilt-slopes
$$
\lim_{\alpha\to 0^{+}} \mu_{\alpha,-1} (\CO_{X}(-2H)[1])=-\frac{1}{2}
< \frac{1}{18}=\lim_{\alpha\to 0^{+}} \mu_{\alpha,-1}(E).
$$
This implies $\Hom(E,\CO_{X}(-2H)[1])=0$.
By Serre duality, we have
\begin{equation}\label{H2-vanish-E-sh}
H^{2}(X,E)\cong \Hom(E,\CO_{X}(-2H)[1])=0.
\end{equation}
By the Hirzebruch--Riemann--Roch theorem, we have 
$$
\chi(E)=\ch_{3}(E)-\frac{1}{6}H^{3} = -h^{1}(X,E)\leq 0
$$ 
and thus $\ch_{3}(E)\leq \frac{1}{6}H^{3}$.
\end{proof}

\begin{rem}\label{on-bundle-UX}
There is a $\mu_{H}$-slope stable sheaf $E$ with Chern character $\ch_{\leq 1}(E)=(4,-H)$  such that $H\cdot \ch_{2}(E)=-\frac{1}{2}H^{3}$.
For example, by Euler exact sequence, 
$\mathfrak{U}_{X}:=\Omega_{\PB^{4}}(H)|_{X}$ is a $\mu_{H}$-slope stable vector bundle with Chern character $\ch(\mathfrak{U}_{X})=(4,-H,-\frac{1}{2}H^{2},-\frac{1}{6}H^{3})$ (see also \cite[\S 3.2]{FLZ23}).
This implies that a $\mu_{H}$-slope stable sheaf with Chern character $v=(4,-H,-\frac{5}{6}H^{2}, \frac{1}{6}H^{3})$ may not be reflexive (see Lemma \ref{may-not-reflexive}).
\end{rem}

Next we show that the Gieseker moduli space $\overline{M}_{X}(v)$ embeds into the corresponding Bridgeland moduli space.

\begin{prop}\label{GieMod-embed-BirMod}
Let $\sigma$ be an arbitrary Serre-invariant stability condition on $\Ku(X)$.
Then there exists an embedding $\overline{M}_{X}(v)\hookrightarrow M_{\sigma}(2[I_{\ell}]+[S(I_{\ell})])$.
In particular, the moduli space $\overline{M}_{X}(v)$ is smooth of dimension $8$.
\end{prop}

\begin{proof}
We first show that if $E$ is a $\mu_{H}$-slope stable sheaf with Chern character $\ch(E)=v$, then we obtain $E\in \Ku(X)$.
In fact, since $E$ is $\mu_{H}$-slope stable and $\mu_{H}(E)=-1/4$,
it follows that $\Hom(\CO_{X},E)=0$ and $\Hom(\CO_{X}(H),E)=0$.
By Serre duality, we have $\Ext^{3}(\CO_{X},E)\cong \Hom(E,\CO_{X}(-2H))=0$ and  
$\Ext^{3}(\CO_{X}(H),E)\cong \Hom(E,\CO_{X}(-H))=0$. 
By the Hirzebruch--Riemann--Roch theorem, we have 
$$
\chi(\CO_{X},E)=0 
\textrm{ and } 
\chi(\CO_{X}(H),E)=0.
$$
It will suffice to show that 
$$
\Ext^{2}(\CO_{X},E)=0
\;  
\textrm{ and }
\;
\Ext^{2}(\CO_{X}(H),E)=0.
$$
It follows from \eqref{H2-vanish-E-sh} that 
$\Ext^{2}(\CO_{X},E) \cong \Hom(E,\CO_{X}(-2H)[1])=0$. 
Since $E$ is a $\mu_{H}$-slope stable sheaf and $\mu_{H}(E)=-\frac{1}{4}>-\frac{9}{10}$,
by Proposition \ref{limit-tilt-stab-Gie-stab-prop} and Lemma \ref{mainwall1-comput}, $E$ is $\sigma_{\alpha,-\frac{9}{10}}$-stable for $\alpha>0$.
Since $\CO_{X}(-H)[1]$ is $\sigma_{\alpha,-\frac{9}{10}}$-stable for $\alpha>0$ and the numerical wall $W(E,\CO_{X}(-H)[1])=W(-\frac{17}{18},\frac{1}{18})$
intersects with the vertical line $\beta=-\frac{9}{10}$,
hence by Serre duality, we have 
$$
\Ext^{2}(\CO_{X}(H),E)\cong \Hom(E, \CO_{X}(-H)[1])=0.
$$
This yields $E\in \Ku(X)$.

For the last statement, for any $H$-Gieseker stable sheaf $E\in \overline{M}_{X}(v)$, 
as the proof of Proposition \ref{non-empty-BSCmoduli}, 
$E$ is $\sigma$-stable and $E\in M_{\sigma}(2[I_{\ell}]+[S(I_{\ell})])$. 
Using the Hirzebruch--Riemann--Roch theorem, we get $\chi(E,E)=-7$.
Since $\Hom(E,E)\cong \CN$, by \cite[Theorem 1.2]{PY22}, we obtain 
$$
\ext^{1}(E,E)=\hom(E,E)-\chi(E,E)=8
$$ 
and $\Ext^{2}(E,E)=0$.
As a consequence, by \cite[Corollary 4.5.2]{HL10}, the moduli space $\overline{M}_{X}(v)$ is smooth of dimension $8$.
\end{proof}

Now we are in the position to finish the proof of Theorem \ref{mainthm-iso}.

\begin{proof}[Proof of Theorem \ref{mainthm-iso}]
Let $\sigma$ be an arbitrary Serre-invariant stability condition on $\Ku(X)$.
By Proposition \ref{GieMod-embed-BirMod},
it is sufficient to show that up to a shift, every $\sigma$-stable object $E\in \Ku(X)$ with numerical class $2[I_{\ell}]+[S(I_{\ell})]$ is isomorphic to a $\mu_{H}$-slope stable sheaf with Chern character $\ch(E)=v$.

By Proposition \ref{one-orbit-unique-Serre-SC}, 
we may assume that $E\in \Ku(X)$ is $\sigma(\alpha,\beta)$-stable object with class $2[I_{\ell}]+[S(I_{\ell})]$ for any $(\alpha,\beta)\in  \mathcal{V}\cap \Gamma$.
By  hypothesis, the Chern character of $E$ is $\ch(E)=\pm v=\pm(4,-H,-\frac{5}{6}H^{2},\frac{1}{6}H^{3})$.
We denote by $\Gamma$ the hyperbola 
\begin{equation*}\label{main-hyperbola}
(\beta+\frac{1}{4})^{2}-\alpha^{2}=\frac{23}{48},
\end{equation*}
which is given by the equation $\mu_{\alpha,\beta}(E)=0$.
Note that $ \mathcal{V}\cap \Gamma\neq \emptyset$ and 
thus $E$ is $\sigma^{0}_{\alpha,\beta}$-semistable, where $(\alpha,\beta)\in  \mathcal{V}\cap \Gamma$.
Thanks to \cite[Proposition 4.1]{FP23}, 
there is either (i) $E$ is $\sigma_{\alpha,\beta}$-semistable; 
or (ii)  there is a distinguished triangle
$$
\xymatrix@C=0.5cm{
F[1] \ar[r]^{} & E \ar[r]^{} & T,}
$$
where $F\in \mathcal{F}_{\alpha,\beta}$ and $T\in \mathcal{T}_{\alpha,\beta}$ are supported at most on points. 

In the case (i), we may assume $\ch(E)=(4,-H,-\frac{5}{6}H^{2},\frac{1}{6}H^{3})$.
By  hypothesis, $E$ is $\sigma_{\alpha,\beta}$-semistable.
By Lemma \ref{mainwall1-comput}, 
there are no walls for objects with $\ch_{\leq 2}(E)=(4,-H,-\frac{5}{6}H^{2})$ to the left of the vertical wall $\beta=-\frac{1}{4}$.
Hence, $E$ is $\sigma_{\alpha,\beta}$-semistable for $\alpha \gg 0$.
By Proposition \ref{limit-tilt-stab-Gie-stab-prop},
$E$ is a $\mu_{H}$-slope stable torsion-free sheaf with Chern character $v$.

In the case (ii), we may assume that $\ch(E)=-(4,-H,-\frac{5}{6}H^{2},\frac{1}{6}H^{3})$.
Since $E$ is $\sigma_{\alpha,\beta}^{0}$-semistable,
it follows that $F$ must be $\sigma_{\alpha,\beta}$-semistable.
Therefore, we may assume that the Chern character of $F$ is
$$
\ch(F)=(4,-H,-\frac{5}{6}H^{2},(\frac{1}{6}+\frac{t}{3})H^{3}),
$$
where $t\geq 0$. 
By Lemma \ref{mainwall1-comput} again,  
it follows that $F$ is $\sigma_{\alpha,\beta}$-semistable for $\alpha\gg0$.
Since $\mu_{H}(F)=-\frac{1}{4}>\beta $, 
by Proposition \ref{limit-tilt-stab-Gie-stab-prop}, $F$ is a $\mu_{H}$-slope stable torsion-free sheaf.
By Proposition \ref{mainsh-ch3-char},  
we have
$\ch_{3}(F)=(\frac{1}{6}+\frac{t}{3})H^{3}\leq \frac{1}{6}H^{3}$ and thus $t=0$, as $t\geq 0$. 
This implies $T=0$ and $E=F[1]$. 
Here $F$ is a  $\mu_{H}$-slope stable torsion-free sheaf with Chern character $\ch(F)=v$.
As a consequence, the embedding $\overline{M}_{X}(v)\hookrightarrow M_{\sigma}(2[I_{\ell}]+[S(I_{\ell})])$ is an isomorphism.
This concludes the proof of Theorem \ref{mainthm-iso}.
\end{proof}

\begin{cor}
The moduli space $\overline{M}_{X}(v)$ is irreducible.
\end{cor}
 
\begin{proof}
According to \cite[Theorem 1.4]{LLPZ24},
the moduli space $M_{\sigma}(2[I_{\ell}]+[S(I_{\ell})])$ is irreducible.
Thus, the corollary follows from Theorem \ref{mainthm-iso}.
\end{proof}

\begin{rem}
Since the Gieseker moduli space $\overline{M}_{X}(v)$ is projective,
Theorem \ref{mainthm-iso} gives a new proof for the projectivity of the moduli space $M_{\sigma}(2[I_{\ell}]+[S(I_{\ell})])$.
\end{rem}

\begin{rem} 
By \cite[Theorem 1.1]{LLPZ24}, the moduli space $M_{\sigma}([I_{\ell}]+2[S(I_{\ell})])$ is non-empty.
Then, to solve the question in Remark \ref{case-rank-5}, one possible approach is to use the same idea as the proof of Theorem \ref{mainthm-iso}.
\end{rem}

\section{Classification of torsion sheaves}\label{Charact-tor-sh}

To prove Theorem \ref{sh-charact-thm}, 
this section gives a classification of tilt semistable torsion sheaves $G$ with Chern character 
$\ch(G)=(0,H,\frac{5}{6}H^{2},-\frac{1}{6}H^{3})$.
First of all, we have the following simple observation.

\begin{lem}\label{ssA-char-lem}
Suppose that $A$ is a $\sigma_{\alpha,\beta}$-semistable object with $\beta\neq 0$ and $\ell\subset X$ is a line. Set $\DD(A):=\CRHom(A,\CO_{X})[1]$. 
\begin{enumerate}
\item[(i)] If $A$ fits into a non-trivial extension 
  $\xymatrix@C=0.3cm{
  0 \ar[r] & \CO_{X}[1] \ar[r] & A \ar[r] & \CO_{\ell} \ar[r] & 0,
  }
  $ 
  then $\DD(A)$ is isomorphic to $I_{\ell}$.
\item[(ii)] If $A$ fits into a non-trivial extension 
  $
  \xymatrix@C=0.3cm{
  0 \ar[r] & \CO_{X}[1] \ar[r] & A \ar[r] &  \CO_{\ell}(-H) \ar[r] & 0,
  }$ then $\CH^{0}(\DD(A))=I_{\ell}$, $\CH^{1}(\DD(A))=\CO_{p}$ for a point $p\in \ell$, 
  and $\CH^{i}(\DD(A))=0$ for $i\neq 0,1$.
\end{enumerate}
\end{lem}

\begin{proof}
In both cases, since $A$ is $\sigma_{\alpha,\beta}$-semistable and $\mu_{\alpha,\beta}(A)\neq \infty$, 
by \cite[Proposition 5.1.3]{BMT14},
there is an exact triangle 
\begin{equation}\label{DDA-SES}
\xymatrix@C=0.5cm{
 \tilde{A} \ar[r] & \DD(A) \ar[r] & T[-1] \ar[r] & \tilde{A}[1],
}
\end{equation}
where $\tilde{A}$ is a $\sigma_{\alpha,-\beta}$-semistable object  and $T$ is a torsion sheaf supported in $0$-dimension. 
Taking cohomology sheaves of \eqref{DDA-SES},  
we have $\CH^{-1}(\tilde{A})\cong \CH^{-1}(\DD(A))$, 
$\CH^{0}(\tilde{A})\cong \CH^{0}(\DD(A))$ and $\CH^{1}(\DD(A))\cong T$.

{\bf Case} (i).
By Serre duality, we have $\Ext^{1}(\CO_{\ell},\CO_{X}[1]) \cong \Ext^{1}(\CO_{X},\CO_{\ell}(-2H))=\CN$.
Thus, the non-trivial extension $A$ is unique.
By Grothendieck--Verdier duality (c.f. \cite[Theorem 3.34]{Huy06}), 
we get 
$$
\CRHom(\CO_{\ell},\CO_{X}) \cong \CRHom(\CO_{\ell},\CO_{\ell}[-2])=\CO_{\ell}[-2].
$$
 Applying $\CRHom(-,\CO_{X})[1]$ to the defining exact sequence of $A$ and taking cohomology sheaves, 
we get $\CH^{-1}(\DD(A))=0$ and an exact sequence
$$
\xymatrix@C=0.5cm{
0 \ar[r] & \CH^{0}(\DD(A)) \ar[r] & \CO_{X} \ar[r]^{\rho_{1}} & \CO_{\ell} \ar[r] &  \CH^{1}(\DD(A)) \ar[r] & 0.
}
$$ 
Since $ \CH^{1}(\DD(A))\cong T$ is a torsion sheaf supported on points,
thus $\rho_{1}$ must be non-zero and then it is surjective.
Hence, we get $T=\CH^{1}(\DD(A))=0$ and so $\CH^{0}(\DD(A))=I_{\ell}$.
This implies $\DD(A)\cong \tilde{A}\cong I_{\ell}$.
  
{\bf Case} (ii). 
Likewise, we obtain $\CH^{-1}(\DD(A))=0$ and an exact sequence
 $$
\xymatrix@C=0.5cm{
0 \ar[r] & \CH^{0}(\DD(A)) \ar[r] & \CO_{X} \ar[r]^{\rho_{2}\;\;\;} & \CO_{\ell}(H) \ar[r] &  \CH^{1}(\DD(A)) \ar[r] & 0,
}
$$ 
where $\rho_{2}$ is also non-zero.
Note that $\rho_{2}$ can not be surjective.
Since $\mathrm{im}(\rho_{2})$ is a structure sheaf, 
so $ \CH^{1}(\DD(A))\cong \CO_{p}$ for some point $p\in X$ and thus $\mathrm{im}(\rho_{2})\cong \CO_{\ell}$.
From this, we deduce $\CH^{0}(\DD(A))=I_{\ell}$.
\end{proof}

Then we will provide some wall analysis in tilt stability for objects with Chern character $(0,H,\frac{5}{6}H^{2},-\frac{1}{6}H^{3})$.

\begin{lem}\label{unique-wall-for-tor-sh}
The semicircle $W(\frac{5}{6},\frac{1}{6})$ is the unique possible wall in tilt stability for objects $G$ with Chern character $\ch_{\leq 2}(G)=(0,H,\frac{5}{6}H^{2})$.
If $\ch_{3}(G)=-\frac{1}{6}H^{3}$ and $G$ is strictly tilt semistable along $W(\frac{5}{6},\frac{1}{6})$, 
then any Jordan--H\"{o}lder filtration of $G$ is given by one of the following cases:
\begin{itemize}
\item
either 
$
\xymatrix@C=0.3cm{
0 \ar[r] & I_{p}(H) \ar[r] & G \ar[r] & A \ar[r] & 0
}
$
or 
$
\xymatrix@C=0.3cm{
0 \ar[r] & A \ar[r] & G \ar[r] & I_{p}(H) \ar[r] & 0,
}
$ 
where $A$ as in Lemma \ref{ssA-char-lem} (i), $p$ is a point and $\ell $ is a line;
\item
either
$
\xymatrix@C=0.3cm{
0 \ar[r] & \CO_{X}(H) \ar[r] & G \ar[r] & A \ar[r] & 0
}
$
or
$
\xymatrix@C=0.3cm{
0 \ar[r] & A \ar[r] & G \ar[r] & \CO_{X}(H) \ar[r] & 0,
}
$
where $A$ as in Lemma \ref{ssA-char-lem} (ii) and $p\in \ell$.
\end{itemize}
\end{lem}

\begin{proof}
Note that all numerical walls for $\ch_{\leq2}(G)=(0,H,\frac{5}{6}H^{2})$ intersect with the vertical line $\beta=\frac{5}{6}$.
Suppose that there is an actual wall induced by the short exact sequence of $\sigma_{\alpha,\frac{5}{6}}$-semistable objects in $\Coh^{\frac{5}{6}}(X)$
\begin{equation}\label{wall-SES-torsion-sh}
\xymatrix@C=0.5cm{
0 \ar[r] & A \ar[r] & G \ar[r] & B\ar[r] & 0.
}
\end{equation}
We write $\ch^{\frac{5}{6}}_{\leq 2}(A):=(a,\frac{b}{6}H,\frac{c}{72}H^{2})$, where $a,b,c\in \ZN$.
Since $0< H^{2}\cdot \ch^{\frac{5}{6}}_{1}(A)=\frac{b}{6}H^{3}<H^{2}\cdot \ch^{\frac{5}{6}}_{1}(G)=H^{3}$,
 we get $b\in \{1,2,3,4,5\}$.
 By changing the roles of $A$ and $B$, 
 we may assume $a\leq -1$.
By direct computations, 
we have the following three possible solutions for \eqref{wall-SES-torsion-sh}:
\begin{enumerate}
  \item[(i)] $b=3$, $a=-3$, $c=-3$, and then $\mathrm{ch}_{\leq  2}(A)=(-3,-2H,-\frac{2}{3}H^{2})$;
  \item[(ii)] $b=4$, $a=-2$, $c=-2$, and then $\mathrm{ch}_{\leq  2}(A)=(-2,-H,-\frac{1}{6}H^{2})$;
  \item[(iii)] $b=5$,  $a=-1$, $c=-1$, and $\ch_{\leq 2}(A)=(-1,0,\frac{1}{3}H^{2})$.
\end{enumerate}
The first two cases are immediately ruled out by Proposition \ref{Li-bound}.

In the following, we will discuss the case (iii) in detail.
In this case, $\ch_{\leq 2}(B)=(1,H,\frac{1}{2}H^{2})$ and the unique possible wall for $G$ is the semicircle $W(\frac{5}{6},\frac{1}{6})$. 
Since $B$ has no walls, so it is a $\mu_{H}$-slope stable sheaf.
Note that $\ch(B(-H))=(1,0,0,\ch_{3}(B)-\frac{1}{6}H^{3})$.
Thanks to \cite[Proposition 4.20]{BBF+}, 
we get $\ch_{3}(B)-\frac{1}{6}H^{3}\leq 0$ and thus $B\cong I_{Z}(H)$, where $Z$ is a zero-dimensional scheme of length $\frac{1}{6}H^{3}-\ch_{3}(B)$. 
Note that $A$ has no walls along $\beta=1$ and $\mu_{H}(A)=0<\frac{5}{6}$.
Since $\beta_{+}(A)<1$ and $W(\frac{5}{6},\frac{1}{6})$ tangent to $\beta=1$,
by \cite[Proposition 4.9]{BBF+}, 
we have $\mathcal{H}^{-1}(A)$ is a $\mu_{H}$-slope stable reflexive sheaf of rank $1$ and
$\mathcal{H}^{0}(A)$ is a torsion sheaf supported in dimension at most $1$.
Hence,  $\mathcal{H}^{-1}(A)$ is a line bundle. 
Since $\ch_{1}(\mathcal{H}^{-1}(A))=0$, 
thus $\mathcal{H}^{-1}(A)\cong \CO_{X}$ and 
$A$ fits into a non-trivial extension
\begin{equation}\label{G-A-def-ses}
\xymatrix@C=0.5cm{
0 \ar[r] & \CO_{X}[1] \ar[r] & A \ar[r] & T_{A} \ar[r] & 0,
}
\end{equation}
where $T_{A}:=\mathcal{H}^{0}(A)$ is a torsion sheaf with $\ch_{\leq 2}(T_{A})=(0,0,\frac{1}{3}H^{2})$. 

If moreover $\ch_{3}(G)=-\frac{1}{6}H^{3}$,
then $\ch_{3}(\CO_{Z}):=\frac{k}{3}H^{3}=\frac{1}{3}H^{3}+\ch_{3}(T_{A})$,
where $k\geq 0$.
Thus, we have $\ch(T_{A})=(0,0,\frac{1}{3}H^{2}, \frac{k-1}{3}H^{3})$ and $\ch(A)=(-1,0,\frac{1}{3}H^{2}, \frac{k-1}{3}H^{3})$.
Since $A$ is tilt semistable, 
by \cite[Theorem 0.1]{Li19} (see e.g.,  \cite[Theorem 2.10]{FP23}), 
the integer $k$ must be $0$ or $1$.
Furthermore, by \cite[Proposition 5.1.3 (b)]{BMT14}, we have $\CExt^{3}(A,\CO_{X})=0$.
By \eqref{G-A-def-ses}, we get $\CExt^{3}(T_{A},\CO_{X})\cong \CExt^{3}(A,\CO_{X})=0$.
Therefore, we have the following two cases:
\begin{itemize}
\item  If $k=1$, then $\ch(T_{A})=(0,0,\frac{1}{3}H^{2}, 0)$ and $Z=p\in X$ is a point.
Since $\CExt^{3}(T_{A},\CO_{X})=0$, 
so the support of $T_{A}$ must be a line $\ell \subset X$ and thus $T_{A}$ is the direct image of a vector bundle on $\ell$. 
According to the Chern character of $T_{A}$, such a vector bundle must be $\CO_{\ell}$ and
thus $T_{A}=\CO_{\ell}$.

\item If $k=0$, then $\ch(T_{A}(H))=(0,0,\frac{1}{3}H^{2}, 0)$ and $I_{Z}=\CO_{X}$. 
Note that $\CExt^{3}(T_{A}(H),\CO_{X})\cong \CExt^{3}(T_{A},\CO_{X})\otimes \CO_{X}(-H)=0$.
As in the case $k=1$, it follows that $T_{A}(H)\cong \CO_{\ell}$ and $T_{A} \cong \CO_{\ell}(-H)$ for a line $\ell\subset X$.
\end{itemize}
As a result, by changing the roles of $A$ and $B$, 
we obtain the following four possible situations:
\begin{enumerate}
\item[(1)] 
$A$ is the unique non-trivial extension
$
\xymatrix@C=0.3cm{
0 \ar[r] & \CO_{X}[1] \ar[r] & A \ar[r] & \CO_{\ell}\ar[r] & 0
}
$
and $G$ fits into the short exact sequence
$
\xymatrix@C=0.3cm{
0 \ar[r] & I_{p}(H) \ar[r] & G \ar[r] & A \ar[r] & 0;
}
$
\item[(2)]  
$A$ is the unique non-trivial extension
$
\xymatrix@C=0.3cm{
0 \ar[r] & \CO_{X}[1] \ar[r] & A \ar[r] & \CO_{\ell}\ar[r] & 0
}
$
and $G$ fits into the short exact sequence
$
\xymatrix@C=0.43cm{
0 \ar[r] & A \ar[r] & G \ar[r] & I_{p}(H) \ar[r] & 0;
}
$
\item[(3)] 
$A$ fits into a non-trivial extension
$
\xymatrix@C=0.3cm{
0 \ar[r] & \CO_{X}[1] \ar[r] & A \ar[r] & \CO_{\ell}(-H)\ar[r] & 0
}
$
and $G$ fits into the short exact sequence
$
\xymatrix@C=0.3cm{
0 \ar[r] & \CO_{X}(H) \ar[r] & G \ar[r] & A \ar[r] & 0;
}
$
\item[(4)]  
$A$ fits into a non-trivial extension
$
\xymatrix@C=0.3cm{
0 \ar[r] & \CO_{X}[1] \ar[r] & A \ar[r] & \CO_{\ell}(-H)\ar[r] & 0
}
$
and $G$ fits into the short exact sequence
$
\xymatrix@C=0.3cm{
0 \ar[r] & A \ar[r] & G \ar[r] & \CO_{X}(H) \ar[r] & 0.
}
$
\end{enumerate}
This completes the proof.
\end{proof}

Next, we have the characterization for the tilt stable objects.

\begin{prop}\label{torsion-sh-stable-case}
If $G$ is a $\sigma_{\alpha,\beta}$-stable object for all $(\alpha,\beta)$ with Chern character $\ch(G)=(0,H,\frac{5}{6}H^{2},-\frac{1}{6}H^{3})$,
then $G\cong \CO_{Y}(D)$ for some Weil divisor $D$ on a hyperplane section $Y\in |H|$.
\end{prop}

\begin{proof}
Since $G$ is $\sigma_{\alpha,\beta}$-stable, 
by \cite[Proposition 5.1.3]{BMT14}, 
there is an exact triangle 
\begin{equation}\label{dual-ss-stable}
\xymatrix@C=0.5cm{
\tilde{G} \ar[r] & \DD(G) \ar[r] &T[-1] \ar[r] &  \tilde{G}[1],
}
\end{equation}
where $\DD(G):=\CRHom(G,\CO_{X})[1]$, $T$ is a torsion sheaf supported in dimension zero and $\tilde{G}\in \Coh^{-\beta}(X)$ is $\sigma_{\alpha,-\beta}$-semistable.
We may assume $\ch_{3}(T)=\frac{s}{3}H^{3}$, where $s\geq 0$.
Then we have $\ch(\tilde{G})=(0,H,-\frac{5}{6}H^{2},(-\frac{1}{6}+\frac{s}{3})H^{3})$ and thus $\tilde{G}$ is a pure sheaf on a hyperplane section $Y\in |H|$.
Then the twisted Chern character $\ch(\tilde{G}(H))=(0,H,\frac{1}{6}H^{2},(-\frac{1}{2}+\frac{s}{3})H^{3})$.
Since $\tilde{G}(H)$ is tilt semistable, according to Lemma \ref{appendix-unique-wall1-prop}, we have $-\frac{1}{2}+\frac{s}{3}\leq -\frac{1}{6}$.
Thus, we get $s=0$ or $s=1$.

If $s=0$, then $\DD(G) \cong \tilde{G}$ is a sheaf and thus $\mathcal{E}xt^{q}(G,\CO_{X})=0$ for $q>1$.
Then \cite[Proposition 1.1.10]{HL10} yields that $G$ is reflexive and then $G\cong \CO_{Y}(D)$ for some Weil divisor $D$ on $Y$ (see e.g.,  \cite[Proposition 3.2]{BBF+}).

If $s=1$, then $\ch(\tilde{G}(H))=(0,H,\frac{1}{6}H^{2}, -\frac{1}{6} H^{3})$.
Note that $\tilde{G}(H)$ is $\sigma_{\alpha,0}$-stable for $\alpha>0$.
Then,  by Lemma \ref{appendix-unique-wall1-prop} below, 
we have a short exact sequence  
\begin{equation}\label{case-s=1}
\xymatrix@C=0.5cm{
0 \ar[r] & I_{\ell}(H) \ar[r] & \tilde{G}(H) \ar[r] &  \CO_{X}[1] \ar[r] & 0,
}
\end{equation}
where $\ell$ is a line in $X$.
Let $A$ be the unique extension as in Lemma \ref{ssA-char-lem} (i).
Then $\DD(A) \cong I_{\ell}$.
Since $\Hom(I_{\ell},T[-2])=0=\Hom(I_{\ell},T[-1])$,
applying $\Hom(I_{\ell},-)$ to \eqref{dual-ss-stable},  
we get 
$$
\Hom(I_{\ell},\tilde{G})\cong \Hom(I_{\ell},\DD(G)).
$$
Thus, by \eqref{case-s=1}, 
we obtain a non-trivial morphism $I_{\ell}\rightarrow \DD(G)$.
Taking duality, we get a non-trivial morphism $G \rightarrow \DD(I_{\ell})=A$.
By Lemma \ref{unique-wall-for-tor-sh}, $G$ is tilt unstable along $W(\frac{5}{6},\frac{1}{6})$, 
a contradiction.
\end{proof}

Next, we will prove Lemma \ref{appendix-unique-wall1-prop} which is used in the proof of Proposition \ref{torsion-sh-stable-case}.
For this, 
we need the following:

\begin{lem}\label{appendix-unique-wall1-lem}
The wall $W(\frac{1}{6},\frac{1}{6})$ is the unique wall in tilt stability for objects $G$ with Chern character $\ch_{\leq 2}(G)=(0,H,\frac{1}{6}H^{2})$.
If  $G$ is strictly semistable along $W(\frac{1}{6},\frac{1}{6})$, 
then $\ch_{3}(G)\leq -\frac{1}{6}H^{3}$ and any Jordan--H\"{o}lder filtration of $E$ is given by either 
$$
\xymatrix@C=0.5cm{
0 \ar[r] & I_{Z}(H) \ar[r] & G \ar[r] &  \CO_{X}[1] \ar[r] & 0,}
\;\;
\textrm{ or }
\;\;
\xymatrix@C=0.5cm{
0 \ar[r] & \CO_{X}[1] \ar[r] & G \ar[r] &  I_{Z}(H) \ar[r] & 0,}
$$
where $Z$ is a line with $\ch_{3}(G)+\frac{1}{6}H^{3}$ embedded/floating
points.
\end{lem}

\begin{proof}
Note that all numerical walls for $\ch_{\leq  2}(G)=(0,H,\frac{1}{6}H^{2})$ intersect with the vertical line $\beta=\frac{1}{6}$. 
Suppose that there is an actual wall induced by the short exact sequence of $\sigma_{\alpha,\frac{1}{6}}$-semistable objects in $\Coh^{\frac{1}{6}}(X)$
$$
\xymatrix@=0.5cm{
0 \ar[r] & A \ar[r] & G \ar[r] & B \ar[r] & 0.
}
$$
We write $\ch_{\leq  2}^{\frac{1}{6}}(A):=(a,\frac{b}{6}H,\frac{c}{72}H^{2})$ with $a,b,c \in \mathbb{Z}$. Since $0<H^{2} \cdot \ch_{1}^{\frac{1}{6}}(A)=\frac{b}{6}H^{3}<H^{2} \cdot \ch_{1}^{\frac{1}{6}}(G)=H^{3}$, we get $b\in \{1,2,3,4,5\}$. By direct computations, we have the following three possible cases:
\begin{enumerate}
\item[(1)] If $b=1$, $a=-1$, $c=-1$, then $\ch_{\leq  2}(A)=(-1,0,0)$ and $\ch_{\leq  2}(B)=(1,H,\frac{1}{6}H^{2})$.
\item[(2)] If $b=2$, $a=-2$, $c=-2$, then $\ch_{\leq  2}(A)=(-2,0,0)$ and $\ch_{\leq  2}(B)=(2,H,\frac{1}{6}H^{2})$.
\item[(3)] If $b=3$,  $a=-3$, $c=-3$, then $\ch_{\leq  2}(A)=(-3,0,0)$ and $\ch_{\leq  2}(B)=(3,H,\frac{1}{6}H^{2})$.
\end{enumerate}
The last two cases are ruled out by Proposition \ref{Li-bound}.

Suppose now $b=1$, $a=-1$, $c=-1$.  
This determines the unique wall $W(\frac{1}{6},\frac{1}{6})$.
In this case, $\alpha=\frac{1}{6}$, $\ch_{\leq  2}(A)=(-1,0,0)$ and $\ch_{\leq  2}(B)=(1,H,\frac{1}{6}H^{2})$.
Then, by \cite[Proposition 4.20]{BBF+}, we get $A\cong \CO_{X}[1]$.
Since the wall $W(\frac{1}{6},\frac{1}{6})$ is tangent to the vertical line $\beta=0$ and $B$ has no walls along $\beta=0$,
by Proposition \ref{limit-tilt-stab-Gie-stab-prop}, $B$ is a $\mu_{H}$-slope stable sheaf.
Twisting $B$ by $\CO_{X}(-H)$, we get 
$
\ch(B(-H))=(1,0,-\frac{1}{3}H^{2},\ch_{3}(G)+\frac{1}{6}H^{3}).
$
Then $B(-H)$ is an ideal sheaf of a curve (may not pure dimension) $Z$.
Hence, the Chern character $\ch(\CO_{Z})=(0,0,\frac{1}{3}H^{2},-\ch_{3}(E)-\frac{1}{6}H^{3})$.
Therefore, we have $\ch_{3}(G)\leq -\frac{1}{6}H^{3}$.
This completes the proof.
\end{proof}
 
Now we prove:
 
\begin{lem}\label{appendix-unique-wall1-prop}
If $G$ is a $\sigma_{\alpha,\beta}$-semistable object with $\ch_{\leq 2}(G)=(0,H,\frac{1}{6}H^{2})$,
then $\ch_{3}(G)\leq -\frac{1}{6}H^{3}$.
If moreover $\ch_{3}(G)=-\frac{1}{6}H^{3}$ and $(\alpha,\beta)$ above the wall $W(\frac{1}{6},\frac{1}{6})$,
then $G$ fits into a non-trivial extension
$$
\xymatrix@C=0.5cm{
0 \ar[r] & I_{\ell}(H) \ar[r] & G \ar[r] &  \CO_{X}[1] \ar[r] & 0,}
$$
where $\ell \subset X$ is a line.
\end{lem}
 
\begin{proof}
Suppose $\ch_{3}(G)\geq -\frac{1}{6}H^{3}$.
It follows from Lemma \ref{appendix-unique-wall1-lem}
that the only possible wall for $G$ is $W(\frac{1}{6},\frac{1}{6})$.
This means that $G$ is tilt semistable along $W(\frac{1}{6},\frac{1}{6})$.
Note that the numerical wall $W(G,\CO_{X}(2H))=W(\frac{1}{6},\frac{11}{6})$ lies above $W(\frac{1}{6},\frac{1}{6})$.
Then both $\CO_{X}(2H)$ and $G$ are tilt stable along $W(G,\CO_{X}(2H))$.
Thus, by Serre duality, we get 
$$
\Hom(G,\CO_{X}[3])\cong \Hom(\CO_{X}(2H),G)=0.
$$
Then, by the Hirzebruch--Riemann--Roch theorem, we have 
$$
\Hom(G,\CO_{X}[1])\geq -\chi(G,\CO_{X})=\ch_{3}(G)+\frac{1}{2}H^{3}\geq 1.
$$
By Lemma \ref{appendix-unique-wall1-lem}, $W(\frac{1}{6}, 
\frac{1}{6})$ is an actual wall for $G$ and thus $\ch_{3}(G)=-\frac{1}{6}H^{3}$.
In this case, the curve $Z$ in Lemma \ref{appendix-unique-wall1-lem} is indeed a line $\ell\subset X$.
\end{proof}

Finally, we obtain the following classification of torsion sheaves.

\begin{thm}\label{Classification-torison-sh}
The semicircle $W(\frac{5}{6},\frac{1}{6})$ is the unique actual wall in tilt stability for objects $G$ with Chern character $\ch(G)=(0,H,\frac{5}{6}H^{2},-\frac{1}{6}H^{3})$.
\begin{enumerate}
\item[(1)] Above $W(\frac{5}{6},\frac{1}{6})$, we have exactly the following three types of objects:
\begin{enumerate}
\item[(i)] $G$ fits into a non-trivial extension
$
\xymatrix@C=0.3cm{
0 \ar[r]^{} & I_{p}(H) \ar[r]^{} & G \ar[r]^{} & A \ar[r]^{} & 0,}
$
where $A$ as in Lemma \ref{ssA-char-lem} (i), 
$\ell$ is a line and $p$ is a point, except for the object $G$ in Remark \ref{rem-for-class-torsion-sh-thm}.               
\item[(ii)] $G$ fits into a non-trivial extension
$
\xymatrix@C=0.3cm{
0 \ar[r]^{} & \CO_{X}(H) \ar[r]^{} & G \ar[r]^{} & A \ar[r]^{} & 0,}
$
where $A$ as in Lemma \ref{ssA-char-lem} (ii) and $p\in \ell$.         
\item[(iii)] $G=\CO_{Y}(D)$, where $D$ is a Weil divisor on $Y\in |H|$.
\end{enumerate}
\item[(2)] Below $W(\frac{5}{6},\frac{1}{6})$,  we have exactly the following three types of objects:
\begin{enumerate}
\item[(i)] $G$ fits into a non-trivial extension
$
\xymatrix@C=0.3cm{
0 \ar[r]^{} & A \ar[r]^{} & G \ar[r]^{} & I_{p}(H) \ar[r]^{} & 0,}
$
where $A$ as in Lemma \ref{ssA-char-lem} (i), 
$\ell$ is a line and $p$ is a point. 
\item[(ii)] $G$ fits into a non-trivial extension
$
\xymatrix@C=0.3cm{
0 \ar[r]^{} & A \ar[r]^{} & G \ar[r]^{} & \CO_{X}(H) \ar[r]^{} & 0,}
$
where $A$ as in Lemma \ref{ssA-char-lem} (ii) and $p\in \ell$, except for the object $G$ in Remark \ref{rem-for-class-torsion-sh-thm}.
\item[(iii)] $G=\CO_{Y}(D)$, where $D$ is a Weil divisor on $Y\in |H|$.
\end{enumerate}
\end{enumerate}
\end{thm}

\begin{proof}
By Proposition \ref{torsion-sh-stable-case}, 
we may assume that $G$ is strictly tilt semistable along $W(\frac{5}{6},\frac{1}{6})$.
By Lemma \ref{unique-wall-for-tor-sh}, we only need to discuss four cases:

{\bf Case 1}. 
We suppose that $G$ fits into a non-trivial extension 
\begin{equation}\label{torsion-sh-class-eq1}
\xymatrix@C=0.5cm{
0 \ar[r]^{} & \CO_{X}(H) \ar[r]^{} & G \ar[r]^{} & A \ar[r]^{} & 0,}
\end{equation}
where $A$ fits into a non-trivial extension
$\xymatrix@C=0.3cm{
0 \ar[r]^{} & \CO_{X}[1] \ar[r]^{} & A \ar[r]^{} & \CO_{\ell}(-H) \ar[r]^{} & 0.}
$
Clearly, $G$ is tilt unstable below $W(\frac{5}{6},\frac{1}{6})$.
Suppose that $G$ is also tilt unstable above $W(\frac{5}{6},\frac{1}{6})$.
By Lemma \ref{unique-wall-for-tor-sh}, 
there exists a destabilizing exact sequence 
\begin{equation}\label{torsion-sh-class-eq1-1}
\xymatrix@C=0.5cm{
0 \ar[r]^{} & F \ar[r]^{} & G \ar[r]^{} & Q \ar[r]^{} & 0,
}
\end{equation}
where there are two possible cases: 
\begin{enumerate}
\item[(i)] $Q\cong \CO_{X}(H)$ and
$
\xymatrix@C=0.3cm{
0 \ar[r]^{} & \CO_{X}[1] \ar[r]^{} & F \ar[r]^{} & \CO_{\ell}(-H) \ar[r]^{} & 0
}
$ is a non-trivial extension;
\item[(ii)] $Q\cong I_{p}(H)$ and
$
\xymatrix@C=0.3cm{
0 \ar[r]^{} & \CO_{X}[1] \ar[r]^{} & F \ar[r]^{} & \CO_{\ell} \ar[r]^{} & 0
}
$  is the unique non-trivial extension.
\end{enumerate}
In case (i), since $\Hom(\CO_{X}(H),\CO_{X}(H))=\CN$ and $\Hom(\CO_{X}(H),F)=0$, 
so \eqref{torsion-sh-class-eq1} splits, a contradiction.
In case (ii),
taking cohomology sheaves of \eqref{torsion-sh-class-eq1} and \eqref{torsion-sh-class-eq1-1},
we get two short exact sequences
$$
\xymatrix@C=0.4cm{
0 \ar[r]^{} & \CO_{X}(H) \ar[r]^{} & \CH^{0}(G)  \ar[r]^{} & \CO_{\ell}(-H)  \ar[r]^{} & 0}
\,
\textrm{ and }
\,
\xymatrix@C=0.4cm{
0 \ar[r]^{} & \CO_{\ell} \ar[r]^{} & \CH^{0}(G)  \ar[r]^{} & I_{p}(H)  \ar[r]^{} & 0.}
$$
Note that $\Hom(\CO_{\ell},\CO_{\ell}(-H))=0$. 
Thus, we get a non-trivial morphism $\CO_{\ell}\rightarrow \CO_{X}(H)$, a contradiction.

{\bf Case 2}. 
We suppose that $G$ fits into a non-trivial extension 
\begin{equation}\label{torsion-sh-class-eq2}
\xymatrix@C=0.5cm{
0 \ar[r]^{} &  A \ar[r]^{} & G \ar[r]^{} & \CO_{X}(H) \ar[r]^{} & 0,}
\end{equation}
where $A$ fits into a non-trivial extension
$
\xymatrix@C=0.4cm{
0 \ar[r]^{} & \CO_{X}[1] \ar[r]^{} & A \ar[r]^{} & \CO_{\ell}(-H) \ar[r]^{} & 0.}
$
Then we have $\CH^{-1}(G)=\CO_{X}$ and $\Ext^{1}(\CO_{X}(H),A)\cong \Ext^{1}(\CO_{X}(H),\CO_{\ell}(-H))=\CN$.
Clearly, $G$ is tilt unstable above $W(\frac{5}{6},\frac{1}{6})$.
Assume that $G$ is also tilt unstable below $W(\frac{5}{6},\frac{1}{6})$.
By Lemma \ref{unique-wall-for-tor-sh}, 
there exists a destabilizing exact sequence 
$
\xymatrix@C=0.3cm{
0 \ar[r]^{} & Q \ar[r]^{} & G \ar[r]^{} & F \ar[r]^{} & 0,
}
$
where there are two cases:
\begin{enumerate}
\item[(i)] $Q\cong \CO_{X}(H)$ and
$
\xymatrix@C=0.3cm{
0 \ar[r]^{} & \CO_{X}[1] \ar[r]^{} & F \ar[r]^{} & \CO_{\ell}(-H) \ar[r]^{} & 0
}
$ is a non-trivial extension;
\item[(ii)] $Q\cong I_{p}(H)$ and
$
\xymatrix@C=0.3cm{
0 \ar[r]^{} & \CO_{X}[1] \ar[r]^{} & F \ar[r]^{} & \CO_{\ell} \ar[r]^{} & 0
}
$  is the unique non-trivial extension.
\end{enumerate}
In case (i), the exact sequence \eqref{torsion-sh-class-eq2} splits, a contradiction.
In case (ii),
taking cohomology sheaves,
we have two short exact sequences
\begin{equation}\label{torsion-sh-class-eq2-1}
\xymatrix@C=0.5cm{
0 \ar[r]^{} & I_{p}(H) \ar[r]^{} & \CH^{0}(G)  \ar[r]^{} & \CO_{\ell}  \ar[r]^{} & 0
}
\end{equation}
and
\begin{equation}\label{torsion-sh-class-eq2-2}
\xymatrix@C=0.5cm{
0 \ar[r]^{} & \CO_{\ell}(-H) \ar[r]^{} & \CH^{0}(G)  \ar[r]^{} & \CO_{X}(H)  \ar[r]^{} & 0.
}
\end{equation}
Therefore, by Remark \ref{rem-for-class-torsion-sh-thm}, we only need to discuss three cases:
\begin{enumerate}
  \item[(a)] If $p\notin \ell$, then $\CH^{0}(G)\cong I_{p}(H)\oplus \CO_{\ell}$.
  Then we have $\CO_{X}(H)\cong \CO_{q}\oplus I_{p}(H)$ for some point $q\in \ell$, a contradiction.
  \item[(b)] If $p\in \ell$ and $\CH^{0}(G)$ is the trivial extension in \eqref{torsion-sh-class-eq2-1}, 
  then we have { $\CO_{X}(H)\cong \CO_{q}\oplus I_{p}(H)$, for some point $q\in \ell$,} a contradiction.
  \item[(c)] Suppose that $p\in \ell$ and $\CH^{0}(G)$ is the non-trivial extension in \eqref{torsion-sh-class-eq2-1}. If $\CH^{0}(G)$ is the trivial extension in \eqref{torsion-sh-class-eq2-2}, then $\CH^{0}(G)\cong \CO_{\ell}(-H) \oplus \CO_{X}(H)$. 
  Using \eqref{torsion-sh-class-eq2-1}, we get $\CO_{p}\oplus \CO_{\ell}(-H)\cong \CO_{\ell}$, a contradiction. 
\end{enumerate}

{\bf Case 3}.  
We suppose that $G$ fits into a non-trivial extension 
\begin{equation}\label{torsion-sh-class-eq3}
\xymatrix@C=0.5cm{
0 \ar[r]^{} & I_{p}(H) \ar[r]^{} & G \ar[r]^{} & A \ar[r]^{} & 0,}
\end{equation}
where $A$ fits into a non-trivial extension
$
\xymatrix@C=0.35cm{
0 \ar[r]^{} & \CO_{X}[1] \ar[r]^{} & A \ar[r]^{} & \CO_{\ell} \ar[r]^{} & 0.
}
$
Clearly, $G$ is tilt unstable below $W(\frac{5}{6},\frac{1}{6})$.
Suppose that $G$ is also tilt unstable above $W(\frac{5}{6},\frac{1}{6})$.
By Lemma \ref{unique-wall-for-tor-sh}, 
there exists a destabilizing exact sequence 
$
\xymatrix@C=0.3cm{
0 \ar[r]^{} & F \ar[r]^{} & G \ar[r]^{} & Q \ar[r]^{} & 0,
}
$
where there are two cases:
\begin{enumerate}
\item[(i)] $Q\cong I_{p}(H)$ and
$
\xymatrix@C=0.3cm{
0 \ar[r]^{} & \CO_{X}[1] \ar[r]^{} & F \ar[r]^{} & \CO_{\ell} \ar[r]^{} & 0
}
$ is the unique non-trivial extension;
\item[(ii)] $Q\cong \CO_{X}(H)$ and
$
\xymatrix@C=0.3cm{
0 \ar[r]^{} & \CO_{X}[1] \ar[r]^{} & F \ar[r]^{} & \CO_{\ell}(-H) \ar[r]^{} & 0
}
$  is a non-trivial extension.
\end{enumerate}
In case (i), 
since $\Hom(I_{p}(H),I_{p}(H))=\CN$ and $\Hom(A,I_{p}(H))=0$, 
so the exact sequence \eqref{torsion-sh-class-eq3} splits, a contradiction.
In case (ii),  the situation is reduced to case (ii) in {\bf Case 2}.

{\bf Case 4}.  
We suppose that $G$ fits into a non-trivial extension 
\begin{equation}\label{torsion-sh-class-eq4}
\xymatrix@C=0.5cm{
0 \ar[r]^{} & A \ar[r]^{} & G \ar[r]^{} & I_{p}(H) \ar[r]^{} & 0,}
\end{equation}
where $A$ fits into a non-trivial extension
$
\xymatrix@C=0.35cm{
0 \ar[r]^{} & \CO_{X}[1] \ar[r]^{} & A \ar[r]^{} & \CO_{\ell} \ar[r]^{} & 0.
}
$
Then we have $\CH^{-1}(G)=\CO_{X}$.
Clearly, $G$ is tilt unstable above $W(\frac{5}{6},\frac{1}{6})$.
Suppose that $G$ is also tilt unstable below $W(\frac{5}{6},\frac{1}{6})$.
By Lemma \ref{unique-wall-for-tor-sh},  
there exists a destabilizing exact sequence
$
\xymatrix@C=0.3cm{
0 \ar[r]^{} & Q \ar[r]^{} & G \ar[r]^{} & F \ar[r]^{} & 0,
}
$
where there are two cases:
\begin{enumerate}
\item[(i)] $Q\cong I_{p}(H)$ and
$
\xymatrix@C=0.3cm{
0 \ar[r]^{} & \CO_{X}[1] \ar[r]^{} & F \ar[r]^{} & \CO_{\ell} \ar[r]^{} & 0
}
$ is the unique non-trivial extension;
\item[(ii)] $Q\cong \CO_{X}(H)$ and
$
\xymatrix@C=0.3cm{
0 \ar[r]^{} & \CO_{X}[1] \ar[r]^{} & F \ar[r]^{} & \CO_{\ell}(-H) \ar[r]^{} & 0
}
$  is a non-trivial extension.
\end{enumerate}
In case (i), the exact sequence \eqref{torsion-sh-class-eq4} splits, a contradiction.
In case (ii), the situation is reduced to case (ii) in {\bf Case 1}.
 \end{proof}

\begin{rem}\label{rem-for-class-torsion-sh-thm}
Let $p\in \ell\subset X$.
Using the short exact sequence
\begin{equation}\label{Ip-Op-SES}
\xymatrix@C=0.5cm{
0 \ar[r]^{} & I_{p}(H) \ar[r]^{} & \CO_{X}(H) \ar[r]^{} & \CO_{p} \ar[r]^{} & 0,
}
\end{equation}
we get $\Ext^{1}(\CO_{\ell},I_{p}(H))\cong \Hom(\CO_{\ell},\CO_{p})=\CN$.
By Serre duality, we have 
$$
\Ext^{1}(\CO_{X}(H),\CO_{\ell}(-H))\cong H^{1}(\PB^{1},\CO(-2))=\CN.
$$
For the unique non-trivial extensions 
$$
\xymatrix@C=0.5cm{
0 \ar[r]^{} & I_{p}(H) \ar[r]^{} & G_{1}  \ar[r]^{} & \CO_{\ell}  \ar[r]^{} & 0
}
\textrm{ and }
\xymatrix@C=0.5cm{
0 \ar[r]^{} & \CO_{\ell}(-H) \ar[r]^{} & G_{2}  \ar[r]^{} & \CO_{X}(H)  \ar[r]^{} & 0,
}
$$
by \eqref{Ip-Op-SES} and the unique non-trivial extension 
$$
\xymatrix@C=0.5cm{
0 \ar[r]^{} & \CO_{\ell}(-H) \ar[r]^{} & \CO_{\ell} \ar[r]^{} & \CO_{p} \ar[r]^{} & 0,
}
$$
we get $G_{1}\cong G_{2}$.
Note that $\Ext^{1}(G_{1},\CO_{X}[1])=\CN^{2}$.
Consider the object $G$ defined by a non-trivial extension 
$
\xymatrix@C=0.3cm{
0 \ar[r]^{} & \CO_{X}[1] \ar[r]^{} & G \ar[r]^{} & G_{1} \ar[r]^{} & 0
}
$
with $\CH^{-1}(\DD(G))=\CO_{X}(-H)$, $\CH^{0}(\DD(G))=I_{\ell}$ and $\CH^{1}(\DD(G))=\CO_{p}$.
\end{rem}

As a consequence, we have the following:

\begin{cor}\label{tor-sh-char-cor}
Suppose that $D$ is a Weil divisor on a hyperplane section $Y\in |H|$ with Chern character $\ch(\CO_{Y}(D)) = (0,H,\frac{5}{6} H^{2},-\frac{1}{6} H^{3})$. 
Then $h^{0}(Y,\CO_{Y}(D))\geq  4$.
\end{cor}

\begin{proof}
According to Theorem \ref{Classification-torison-sh}, 
the sheaf $\CO_{Y}(D)$ is $\sigma_{\alpha,\beta}$-stable for all $\alpha >0$ and $\beta \in \RN$.  
{By} \cite[{Proposition 2.14}]{BLMS23}, {$\CO_{X}(-2H)[1]$ is $\sigma_{\alpha,\beta}$-stable for $\beta\geq -2$.
By direct computations, 
the numerical wall $W(\CO_{Y}(D),\CO_{X}(-2H)[1])=W(\frac{5}{6},\frac{17}{6})$ is non-empty. }
Hence, by Serre duality, we get 
$$
H^{2}(Y,\CO_{Y}(D))\cong \Hom(\CO_{Y}(D),\CO_{X}(-2H)[1])=0.
$$ 
By the Hirzebruch--Riemann--Roch theorem,
we have
$$
h^0(Y,\CO_{Y}(D)) \geq \chi(\CO_{Y}(D)) = 4.
$$
This finishes the proof.
\end{proof}


\section{Proof of Theorem \ref{sh-charact-thm}}
\label{Proof-main-2result}

The goal of this section is to use wall-crossing techniques to prove Theorem \ref{sh-charact-thm}, 
which gives a set-theoretic description of the Gieseker moduli space $\overline{M}_{X}(v)$,
where $v=2\ch(I_{\ell})+\ch(S(I_{\ell}))=(4,-H,-\frac{5}{6}H^{2},\frac{1}{6}H^{3})$.


\subsection{Non-reflexive stable sheaves}
It is worth noticing that every $\mu_{H}$-slope stable sheaf with Chern character $\ch(I_{\ell})+\ch(S(I_{\ell}))$ is reflexive (see \cite[Proposition 6.6]{BBF+}). 
However, in our case, we will show that there exist non-reflexive $\mu_{H}$-slope stable sheaves with Chern character $v=2\ch(I_{\ell})+\ch(S(I_{\ell}))$.
Let us first recall the $\mu_{H}$-slope stable vector bundle $\mathcal{U}_{X}$ in Remark \ref{on-bundle-UX} which is given by the restriction of Euler exact sequence on $\PB^{4}$,
\begin{equation}\label{def-of-UX}
\xymatrix@C=0.5cm{
0 \ar[r]^{} & \mathcal{U}_{X} \ar[r]^{} & \CO_{X}^{\oplus 5} \ar[r]^{} & \CO_{X}(H)\ar[r]^{} & 0,
}
\end{equation}
where $\mathcal{U}_{X}:=\Omega_{\mathbb{P}^{4}}(H) |_{X}$.
Then we have the following (see \cite[Lemmas 5.3 and 5.4]{FLZ23}):

\begin{lem}\label{charact-UX}
Let $E$ be a $\mu_{H}$-slope stable sheaf with Chern character $\ch_{\leq 1}(E)=(4,-H)$.
Then $H\cdot \ch_{2}(E)\leq -\frac{1}{2}H^{3}$.
If moreover $H\cdot \ch_{2}(E)=-\frac{1}{2}H^{3}$,
then $\ch_{3}(E)\leq -\frac{1}{6}H^{3}$ and the following conditions are equivalent:
\begin{enumerate}
  \item[(1)] $\ch_{3}(E)=-\frac{1}{6}H^{3}$;
  \item[(2)] $E$ is a reflexive sheaf;
  \item[(3)] $E \cong \mathcal{U}_{X}$.
\end{enumerate}
\end{lem}

\begin{proof}
The maximal second Chern character of $E$ was obtained in \cite[Lemma 5.4]{FLZ23}.
For (1) $\Rightarrow$ (2), it follows from \cite[Lemma 5.3]{FLZ23}.
(3) $\Rightarrow$ (1) is a direct consequence, 
since $\mathcal{U}_{X}$ is $\mu_{H}$-slope stable.
For (2) $\Rightarrow$ (3),  by \cite[Lemma 5.4]{FLZ23},
one has $\ch_{3}(E)\leq -\frac{1}{6}H^{3}$. 
By the Hirzebruch--Riemann--Roch theorem, we have $\chi(E,\CO_{X})\geq 5$. 
Since $E$ is reflexive, combining with \cite[Proposition 4.20 (i)]{BBF+},  (2) $\Rightarrow$ (3) follows from the proof of \cite[Lemma 5.4]{FLZ23}. 
In fact, note that $\Hom(E,\CO_{X}[2])\cong \Hom(\CO_{X}(2H),E[1])=0$.
Thus, we have $\hom(E,\CO_{X})\geq \chi(E,\CO_{X})\geq 5$.
Then there is a short exact sequence of tilt semistable objects
$$
\xymatrix@C=0.5cm{
0 \ar[r]^{} & G \ar[r]^{} & E[1] \ar[r]^{} & \CO_{X}^{\oplus 5}[1]\ar[r]^{} & 0}.
$$ 
Hence, we get $\ch(G)=(1,H,\frac{1}{2}H^{2},-\ch_{3}(E))$.
By \cite[Proposition 4.20 (i)]{BBF+}, we have $$\ch_{3}(G(-H))=-\ch_{3}(E)-\frac{1}{6}H^{3}\leq 0.$$ 
Thus, we derive  $\ch_{3}(E)\geq -\frac{1}{6}H^{3}$. 
It follows that $\ch_{3}(E)=-\frac{1}{6}H^{3}$ and $\ch(G(-H))=(1,0,0,0)$. 
Therefore, we conclude $G=\CO_{X}(H)$ and $E\cong \mathcal{U}_{X}$.
\end{proof}

Then we obtain a simple observation.

\begin{lem}\label{non-reflexive-charact}
Let $\ell \subset X$ be a line.
Then $\Hom(\mathcal{U}_{X},\CO_{\ell}(-H))=\CN$ and a non-zero morphism $\phi_{\ell}: \mathcal{U}_{X}\rightarrow \CO_{\ell}(-H)$ is surjective.
In particular, the kernel sheaf $K_{\ell}$ of $\phi_{\ell}$ is a non-reflexive $\mu_{H}$-slope stable sheaf with Chern character $v$ and $\CExt^{1}(K_{\ell},\CO_{X})\cong \CO_{\ell}(H)$.
\end{lem}

\begin{proof}
Applying $\Hom(-,\CO_{\ell}(-H))$ to the exact sequence \eqref{def-of-UX},
we get 
$$
\Hom(\mathcal{U}_{X},\CO_{\ell}(-H))\cong \Ext^{1}(\CO_{X}(H),\CO_{\ell}(-H))=\CN.
$$
Suppose $\phi_{\ell}: \mathcal{U}_{X}\rightarrow \CO_{\ell}(-H)$ is a non-trivial morphism.
Thus, the kernel sheaf $K_{\ell}:=\ker(\phi_{\ell})$ is unique for each line $\ell\subset X$.
Since $\mathcal{U}_{X}$ is $\mu_{H}$-slope stable and $\CO_{\ell}(-H)$ is supported on the line $\ell$,
we obtain $\mu_{H}(K_{\ell})=-\frac{1}{4}$ and $K_{\ell}$ is also $\mu_{H}$-slope stable.
Then there is an exact sequence
\begin{equation}\label{to-ker-surj}
\xymatrix@C=0.5cm{
0 \ar[r]^{} & K_{\ell} \ar[r]^{} & \mathcal{U}_{X} \ar[r]^{\phi_{\ell}\;\;\;\;\;} & \CO_{\ell}(-H)\ar[r]^{} & C_{\ell}\ar[r]^{} &0,
}
\end{equation}
where the cokernel sheaf $C_{\ell}$ of $\phi_{\ell}$ is supported on the line $\ell$.
We may assume $\ch(C_{\ell})=(0,0,\frac{s}{3}H^{2},\ch_{3}(C_{\ell}))$, where $s\geq 0$.
Then the Chern character
$$
\ch(K_{\ell})=(4,-H,(-\frac{5}{6}+\frac{s}{3})H^{2},\frac{1}{6}H^{3}+\ch_{3}(C_{\ell})).
$$
Since $C_{\ell}$ is supported on the line $\ell$, we derive $s=0$ or $s=1$.
\begin{enumerate}
\item[(1)] If $s=0$, then we may assume $\ch_{3}(C_{\ell})=\frac{t}{3}H^{3}$, where $t\geq 0$.
Hence, we have $\ch(K_{\ell})=(4,-H,-\frac{5}{6}H^{2},\frac{1+2t}{6}H^{3})$. 
By Proposition \ref{mainsh-ch3-char}, 
we get $\frac{1+2t}{6}\leq \frac{1}{6}$. 
Thus, we derive $t=0$ and then $C_{\ell}=0$. 
This means that $\phi_{\ell}$ is surjective.
\item[(2)] If $s=1$, then we get $\ch(C_{\ell})=(0,0,\frac{1}{3}H^{2},\ch_{3}(C_{\ell}))$. 
Thus, by the sequence \eqref{to-ker-surj}, we have $\ch(\mathrm{im}(\phi_{\ell}))=(0,0,0,-\frac{1}{3}H^{3}-\ch_{3}(C_{\ell}))$.
Since $\mathrm{im}(\phi_{\ell})$ is a torsion sheaf supported on $\ell$,
we obtain $\mathrm{im}(\phi_{\ell})$ is a direct image of a torsion sheaf on $\ell$.
This is a contradiction with the fact that $\phi_{\ell}$ is non-trivial.
\end{enumerate}
As a consequence, the morphism $\phi_{\ell}$ is surjective.
Then there is a short exact sequence 
\begin{equation}\label{Kl-embed-UX}
\xymatrix@C=0.5cm{
0 \ar[r]^{} & K_{\ell} \ar[r]^{} & \mathcal{U}_{X} \ar[r]^{\phi_{\ell}\;\;\;\;\;} & \CO_{\ell}(-H)\ar[r]^{} & 0.
}  
\end{equation}
By Grothendieck--Verdier duality, 
we have  
$$
\CExt^{1}(\CO_{\ell}(-H),\CO_{X})=0
\;
\textrm{and}\;
\CExt^{2}(\CO_{\ell}(-H),\CO_{X})\cong \CO_{\ell}(H).
$$
Applying $\CHom(-,\CO_{X})$ to \eqref{Kl-embed-UX}, 
it follows that the double dual $K_{\ell}^{\vee\vee}\cong  \mathcal{U}_{X}$ 
and { $\CExt^{1}(K_{\ell},\CO_{X}) \cong$}
{$ \CExt^{2}(\CO_{\ell}(-H),\CO_{X})\cong \CO_{\ell}(H)$.}
Therefore, $K_{\ell}$ is not reflexive.
\end{proof}

Based on Lemma \ref{non-reflexive-charact}, 
the following result implies that there exist non-reflexive $\mu_{H}$-slope stable sheaves with Chern character $v$.

\begin{prop}\label{may-not-reflexive}
Let $E$ be a $\mu_{H}$-slope stable sheaf  of Chern character $\ch(E)=v$.
Then either $E$ is a reflexive sheaf, 
or $E$ is not reflexive and exactly fits into a short exact sequence 
$$
\xymatrix@C=0.5cm{
0 \ar[r]^{} & E \ar[r]^{} & \mathcal{U}_{X} \ar[r]^{} & \CO_{\ell}(-H)\ar[r]^{} & 0,}
$$
where $\ell\subset X$ is a line.
\end{prop}

\begin{proof}
Since $E$ is a torsion-free sheaf, 
so there exists a short exact sequence
$$
\xymatrix@C=0.5cm{
0 \ar[r]^{} & E \ar[r]^{} & E^{\vee\vee} \ar[r]^{} & T \ar[r]^{} & 0,}
$$
where $T$ is a torsion sheaf supported in codimension $\geq 2$.
We may write the Chern character $\ch(T):=(0,0,\frac{s}{3}H^{2}, \ch_{3}(T))$ with $s \in \ZN_{\geq 0}$.
Note that the dual sheaf of a $\mu_{H}$-slope stable sheaf is also $\mu_{H}$-slope stable.
Hence, the double dual sheaf $E^{\vee \vee}$ is $\mu_{H}$-slope stable. 
By Lemma \ref{charact-UX}, 
we have 
$$
H\cdot \ch_{2}(E^{\vee\vee})=-\frac{5}{6}H^{3}+\frac{s}{3}H^{3} \leq -\frac{1}{2}H^{3}.
$$
Hence, we get $s=0$ or $s=1$.

(1) If $s=0$, we may assume $\ch(T)=(0,0,0,\frac{t}{3}H^{3})$ with $t\in \ZN_{\geq0}$.
By Proposition \ref{mainsh-ch3-char}, we have 
$$
\ch_{3}(E^{\vee\vee})=\ch_{3}(E)+\frac{t}{3}H^{3} \leq \frac{1}{6}H^{3}=\ch_{3}(E).
$$
This yields $\ch(E)=\ch(E^{\vee \vee})$ and $t=0$. 
Thus, we get $T=0$.
Therefore, we conclude $E\cong E^{\vee\vee}$, i.e., $E$ is reflexive.

(2) If  $s=1$, 
then $\ch_{\leq 2}(E^{\vee\vee})=(4,-H,-\frac{1}{2}H^{2})$.
Since $E^{\vee\vee}$ is reflexive, 
by Lemma \ref{charact-UX}, 
we have $E^{\vee\vee}\cong \mathcal{U}_{X}$ and thus $\ch(T)=(0,0,\frac{1}{3}H^{2},-\frac{1}{3}H^{3})$.
From \eqref{def-of-UX}, we have $\Ext^{i}(\CO_{X},\mathcal{U}_{X})=0$ for all $i$.
Since $E\in \Ku(X)$, it follows that $\Hom(\CO_{X},T)=0$.
Twisting $T$ by $\CO_{X}(H)$, 
we have $\ch(T(H))=(0,0,\frac{1}{3}H^{2},0)$.
By the Hirzebruch--Riemann--Roch theorem, we have $\chi(T(H))=1$.
Since $T(H)$ is supported on a curve,
hence we have $\Hom(\CO_{X},T(H))\neq 0$.
Let $\phi\in \Hom(\CO_{X},T(H))$ be a non-trivial morphism.
Then there is an exact sequence
$$
\xymatrix@C=0.5cm{
0 \ar[r]^{} & K \ar[r]^{} & \CO_{X} \ar[r]^{\phi\;\;\;} & T(H) \ar[r]^{} & C \ar[r]^{} & 0,}
$$
where $C$ is the cokernel of $\phi$.
According to the Chern character of $T(H)$, we have $\ch_{2}(C)=0$ or $\ch_{2}(C)=\frac{1}{3}H^{2}$.
\begin{enumerate}
\item[(i)] If $\ch_{2}(C)=0$, then $\ch(\mathrm{im}(\phi))=(0,0,\frac{1}{3}H^{2},-\frac{k}{3}H^{3})$, 
where $k\geq 0$.
Since $\mathrm{im}(\phi)$ is a structure sheaf, 
so $k=0$ and thus $T(H)\cong \mathrm{im}(\phi) \cong \CO_{\ell}$ for a line $\ell\subset X$.  
\item[(ii)]  If $\ch_{2}(C)=\frac{1}{3}H^{2}$, then $\mathrm{im}(\phi)$ is the structure sheaf of points. 
This is a contradiction with $\Hom(\CO_{X},T)=0$.
\end{enumerate}
Therefore, we get $T(H) \cong \CO_{\ell}$ and so $T \cong \CO_{\ell}(-H)$.
From Lemma \ref{non-reflexive-charact}, it follows that the assertion holds.
\end{proof}


\subsection{Characterization of stable sheaves}

We will first give a general construction of $\mu_{H}$-slope stable vector bundles with Chern character $v=(4,-H,-\frac{5}{6}H^{2},\frac{1}{6}H^{3})$.
Recall that the {\it left mutation} of $F\in \DC(X)$ through $\CO_{X}$ is given by the cone of the evaluation morphism 
$$
\mathrm{L}_{\CO_{X}}(F)
:=
\mathrm{Cone}\Big(\RHom(\CO_{X},F)\otimes \CO_{X} \xrightarrow{\mathrm{ev}} F \Big).
$$
In particular, one has $\mathrm{L}_{\CO_{X}}(\CO_{X})=0$.

\begin{prop}\label{more-vector-bundle}
{ 
Let $C\subset X$ be a line-free non-degenerate locally Cohen–Macaulay curve of degree $4$, $p_{a}(C)=0$ and $H^{0}(C,\CO_{C})=\CN$.}
Let $$
E_{C}:=\mathrm{L}_{\CO_{X}} (\DD(I_{C}(H)))[-1].
$$
Then $E_{C}$ is a $\mu_{H}$-slope stable vector bundle with $\ch(E_{C})=v=(4,-H,-\frac{5}{6}H^{2},\frac{1}{6}H^{3})$. 
In particular, $E_{C}\in \Ku(X)$.
\end{prop}

\begin{proof}
Since $C$ is a locally Cohen–Macaulay curve of degree $4$ and arithmetic genus $p_{a}(C)=0$,   
we have $\ch(\CO_{C})=(0,0,\frac{4}{3}H^{3},-H^{3})$ and a short exact sequence
\begin{equation}\label{degree4-curve-stsh}
\xymatrix@C=0.5cm{
0 \ar[r] & I_{C}  \ar[r]^{} & \CO_{X} \ar[r] &  \CO_{C}   \ar[r] & 0.
}
\end{equation}
Then, by \eqref{degree4-curve-stsh} and Hirzebruch–Riemann–Roch theorem, 
we derive $\RHom(\CO_{X}(H),\CO_{C})=\CN^{3}[-1]$.
Since $H^{0}(C,\CO_{C})=\CN$ , 
applying $\RHom(\CO_{X},-)$ and $\RHom(\CO_{X}(H),-)$ to the derived dual of \eqref{degree4-curve-stsh}, 
we have 
$$
\RHom(\CO_{X}(H),\DD(I_{C}(H)))=0
\,
\textrm{ and }
\,
\RHom(\CO_{X},\DD(I_{C}(H)))=\CN^{3}[0].
$$
Hence, by the definition of left mutation, we  get an exact triangle
\begin{equation}\label{left-mu-exact}
\xymatrix@C=0.5cm{
\CO_{X}^{\oplus 3}  \ar[r] & \DD(I_{C}(H)) \ar[r] & E_{C}[1] \ar[r] &  \CO_{X}^{\oplus 3}[1].
}
\end{equation}

Since $C$ is a Cohen–Macaulay curve, the dualizing sheaf $\omega_{C}\cong \CExt^{2}(\CO_{C},\omega_{X})$.
Applying $\CRHom(-,\CO_{X}(-H))$ to \eqref{degree4-curve-stsh}, 
we have an exact triangle
\begin{equation}\label{ext-tri-OC}
\xymatrix@C=0.5cm{
 \CO_{X}(-H)[1] \ar[r]^{} & \DD(I_{C}(H))\ar[r] &  \omega_{C}(H)  \ar[r] & \CO_{X}(-H)[2].
 }
\end{equation}
By using Serre duality and \eqref{ext-tri-OC}, 
we get 
$$
\Ext^{1}(\omega_{C}(H),\CO_{X}(-H)[1])\cong \Ext^{1}(\CO_{X},\omega_{C})\cong H^{0}(C,\CO_{C})=\CN.
$$
{By} \cite[{Proposition 2.14}]{BLMS23}, {$\CO_{X}(-H)[1]$ is $\sigma_{\alpha,\beta}$-stable for $\beta\geq -1$.
Since $H^2\mathrm{ch}_1^\beta(\omega_C(H))=0$, 
$\omega_{C}(H)$ is $\sigma_{\alpha,\beta}$-semistable for $\beta\geq -1$.}

 Then, by \eqref{ext-tri-OC},
we obtain the unique non-trivial extension
$$
\xymatrix@C=0.5cm{
0 \ar[r] & \CO_{X}(-H)[1] \ar[r]^{} & \DD(I_{C}(H)) \ar[r] &  \omega_{C}(H) \ar[r] & 0,}
$$
in $\Coh^{\beta}(X)$.

Since $I_{C}(H)$ is a $\mu_{H}$-slope stable sheaf with $\mu_{H}(I_{C}(H))=1$,
by Proposition \ref{limit-tilt-stab-Gie-stab-prop},  
$I_{C}(H)$ is $\sigma_{\alpha,\beta}$-stable for $\beta<1$ and $\alpha\gg 0$.
Hence, the derived dual $\DD(I_{C}(H))$ is $\sigma_{\alpha,\beta}$-semistable for $\beta>-1$ and $\alpha\gg 0$.
Note that $\ch(\DD(I_{C}(H)))=(-1,H,\frac{5}{6}H^{2},-\frac{1}{6}H^{3})$ and the numerical wall $W(\DD(I_{C}(H)),\CO_{X}^{\oplus 3}[1])=W(\frac{5}{6},\frac{5}{6})$.
Clearly $\DD(I_{C}(H))$ has no walls along $\beta=0$ and $W(\frac{5}{6},\frac{5}{6})$ tangent to $\beta=0$. 
Thus, $\DD(I_{C}(H))$ is tilt semistable (strictly or not, depends on $C$ is degenerate or not) along $W(\frac{5}{6},\frac{5}{6})$.
Since $\CO_{X}^{\oplus 3}[1]$ is also $\sigma_{\alpha,\beta}$-semistable for $\beta\geq 0$ and $\alpha>0$,
it follows from \eqref{left-mu-exact} that there is a short exact sequence
\begin{equation}\label{more-vb-ses}
\xymatrix@C=0.5cm{
0 \ar[r] & \DD(I_{C}(H)) \ar[r]^{} & E_{C}[1] \ar[r] &  \CO_{X}^{\oplus 3}[1]  \ar[r] & 0,}
\end{equation}
in the abelian category of $\sigma_{\alpha,\beta}$-semistable objects of the same slope as $\CO_{X}^{\oplus 3}[1]$ and $\DD(I_{C}(H))$.

We set $\tilde{E}:=E_{C}[1]$.
Thus, the object $\tilde{E}$ is tilt semistable on $W(\frac{5}{6},\frac{5}{6})$ and $\ch(\tilde{E})=-v$.
Then there is a short exact sequence
$$
\xymatrix@C=0.5cm{
0 \ar[r]^{} & \CH^{-1}(\tilde{E})[1] \ar[r]^{} & \tilde{E} \ar[r]^{} & \CH^{0}(\tilde{E}) \ar[r]^{} & 0.}
$$
Note that there are no walls along $\beta=0$ for $\tilde{E}$.
By \cite[Proposition 4.9]{BBF+}, $\CH^{-1}(\tilde{E})$ is $\mu_{H}$-slope stable reflexive and 
$ \CH^{0}(\tilde{E})$ is supported in dimension no more than $1$. 
We may assume $\ch( \CH^{0}(\tilde{E}))=(0,0,\frac{s}{3}H^{2},\ch_{3}( \CH^{0}(\tilde{E})))$, where $s\geq 0$.
Then we have
$$
\ch( \CH^{-1}(\tilde{E}))=(4,-H,(-\frac{5}{6}+\frac{s}{3})H^{2},\frac{1}{6}H^{3}+\ch_{3}( \CH^{0}(\tilde{E}))).
$$
Since $\CH^{-1}(\tilde{E})$ is $\mu_{H}$-slope stable, by Lemma \ref{charact-UX}, 
we have $-\frac{5}{6}+\frac{s}{3}\leq -\frac{1}{2}$ and thus $s=0$ or $s=1$.
We discuss two cases:
\begin{enumerate}
\item[(i)] If $s=0$, then Proposition \ref{mainsh-ch3-char} yields $\CH^{0}(\tilde{E})=0$.
Thus, we obtain $\tilde{E} \cong \CH^{-1}(\tilde{E})[1]$, as we wanted.

\item[(ii)] If $s=1$, then we have $\ch_{\leq 2}( \CH^{-1}(\tilde{E}))=(4,-H,-\frac{1}{2}H^{2})$.
Since $\CH^{-1}(\tilde{E})$ is reflexive, by Lemma \ref{charact-UX}, we get $\CH^{-1}(\tilde{E})\cong \mathcal{U}_{X}$.
Thus, we obtain $\ch(\CH^{0}(\tilde{E}))=(0,0,\frac{1}{3}H^{2},-\frac{1}{3}H^{3})$.
Hence, the support of $\CH^{0}(\tilde{E})$ contains a line.
Considering the exact sequence of the cohomology sheaves of \eqref{more-vb-ses},
we get
\begin{equation}\label{cohomology-sheaves-seq-Prop-6.4}
\xymatrix@C=0.5cm{
0 \ar[r] & \CO_{X}(-H) \ar[r] & \mathcal{U}_{X} \ar[r] & \CO_{X}^{\oplus 3} \ar[r] & \omega_{C}(H) \ar[r]^{\varrho} & \CH^{0}(\tilde{E}) \ar[r] & 0.}
\end{equation}
It follows that the support of $\CH^{0}(\tilde{E})$ belongs to $C$.
Note that $C$ is Cohen-Macaulay and $\ch(\CH^{0}(\tilde{E}))=(0,0,\frac{1}{3}H^{2},-\frac{1}{3}H^{3})$.
By \eqref{def-of-UX}, we get $H^{i}(X,\mathcal{U}_{X})=0$ for all $i\in \ZN$.
By Serre duality, $H^{1}(X,\omega_{C}(H))\cong H^{1}(C,\omega_{C}(H))\cong H^{0}(C,\CO_{C}(-H))=0$.
Note that $\chi(\omega_{C}(H))=3$ and thus $H^{0}(X,\omega_{C}(H))=\CN^{3}$.
By \eqref{cohomology-sheaves-seq-Prop-6.4}, we derive $H^{0}(X,\ker(\varrho))=\CN^{3}$ and $H^{i}(X,\ker(\varrho))=0$ for $i\neq 0$.
This implies $H^{0}(X,\CH^{0}(\tilde{E}))=0$.
Hence, $\CH^{0}(\tilde{E})$ must be a pure sheaf and thus $\CH^{0}(\tilde{E})=\CO_{\ell}(-H)$ for some line $\ell\subset X$.
Since $C$ is line-free (i.e., $\Hom(\CO_{\ell},\CO_{C})=0$), 
so $\Hom( \DD(I_{C}(H)),\CO_{\ell}(-H))\cong \Ext^{1}(\CO_{\ell},I_{C})=0$ and thus
$\Hom( \DD(I_{C}(H)),\mathcal{U}_{X}[1])\neq 0$. 
Applying $\Hom(-,I_{C}(H))$ to the dual of short exact sequence \eqref{def-of-UX}, we derive
$$
\xymatrix@C=0.5cm{
0 \ar[r] & \Hom(\mathcal{U}_{X}^{\vee},I_{C}(H)) \ar[r] & \Hom(\CO_{X},I_{C}(H))^{\oplus 5}  \ar[r] & 
 \Hom(\CO_{X}(-H),I_{C}(H))}.
$$
Since $C$ is non-degenerate, 
so $\Hom(\CO_{X},I_{C}(H)=0$. Thus,
we have
$$
\Hom(\DD(I_{C}(H)),\mathcal{U}_{X}[1])\cong \Hom( \mathcal{U}_{X}^{\vee}, I_{C}(H))\cong \Hom(\CO_{X},I_{C}(H))^{\oplus 5}=0,
$$
a contradiction.

\end{enumerate}
As a consequence, 
we obtain $\tilde{E}=E_{C}[1]\cong  \CH^{-1}(\tilde{E})[1]$.
Hence, $E_{C}=\CH^{-1}(\tilde{E})$ is a $\mu_{H}$-slope stable reflexive sheaf with $\ch(E)=v$.
Applying $\CRHom(-,\CO_{X})$ to \eqref{more-vb-ses}, we get
$\CExt^{i}(E_{C},\CO_{X})=0$ for $i=1,2,3$.
Therefore, $E_{C}$ is a $\mu_{H}$-slope stable vector bundle with $\ch(E_{C})=v$. 
In particular, by the exact triangle \eqref{left-mu-exact}, we have $E_{C}\in \Ku(X)$.
This finishes the proof of Proposition \ref{more-vector-bundle}.
\end{proof}

In order to prove Theorem \ref{sh-charact-thm}, 
based on Corollary \ref{tor-sh-char-cor},
we give a special construction of $\mu_{H}$-slope stable vector bundles with Chern character $v=(4,-H,-\frac{5}{6}H^{2},\frac{1}{6}H^{3})$.

\begin{prop}\label{construt-bundle-converse}
Let $D$ be a Weil divisor on a hyperplane section $Y\in |H|$ with Chern character $\ch(\CO_{Y}(D)) = (0,H,\frac{5}{6} H^{2},-\frac{1}{6} H^{3})$. 
If  $H^{0}(Y,\CO_{Y}(D))\cong \CN^{4}$, 
then $\CO_{Y}(D)$ is globally generated and $E_{D}$ is a $\mu_{H}$-slope stable vector bundle with $\ch(E_{D})=v$. 
In particular, there is a (degenerate) smooth quartic rational curve $C$ in $Y$ of class $D$ and $E_{D}\cong \mathrm{L}_{\CO_{X}}(\DD(I_{C}(H)))[-1]$. 
\end{prop}

\begin{proof}
We set $V:=H^{0}(Y,\CO_{Y}(D))=\CN^{4}$.
Then, as in Section \ref{non-emptiness-Modsp}, 
Lemma \ref{(BBF+)-sh-constrct-lem} yields that
$E_{D, V}$ is a $\mu_{H}$-slope stable reflexive sheaf.
We may assume $\ch(\CH^{0}(\mathcal{E}_{D,V})):= (0,0,\frac{s}{3}H^{2}, \ch_{3}(\CH^{0}(\mathcal{E}_{D,V}))$, 
where $s \geq 0$.
Then, we have
$$
\ch(E_{D,V})=(4,-H, (\frac{s}{3}-\frac{5}{6})H^{2},\ch_{3}(\CH^{0}(\mathcal{E}_{D,V}))+\frac{1}{6}H^{3}).
$$
By Lemma \ref{charact-UX}, 
we obtain $\frac{s}{3}-\frac{5}{6} \leq -\frac{1}{2}$.
Thus, we get $s=0$ or $s=1$.

{\bf Case (i):}  
If $s=1$, then we derive $\ch(E_{D,V})=(4,-H,-\frac{1}{2}H^{2},\ch_{3}(\CH^{0}(\mathcal{E}_{D,V}))+\frac{1}{6}H^{3})$ and $\CH^{0}(E_{D,V})$ is supported on a line $\ell\subset X$.
Since $E_{D,V}$ is reflexive, by Lemma \ref{charact-UX}, we get $E_{D,V}\cong \mathcal{U}_{X}$.
Then we have
$\ch(\CH^{0}(\mathcal{E}_{D,V}))=(0,0,\frac{1}{3}H^{2},-\frac{1}{3}H^{3})$.
Then the exact sequence \eqref{(BBF+)-sh-constrct} gives 
\begin{equation}\label{case-i}
\xymatrix@C=0.5cm{
0 \ar[r] & \mathcal{U}_{X} \ar[r] & \CO_{X}^{\oplus 4} \ar[r] & \CO_{Y}(D) \ar[r] & \CH^{0}(\mathcal{E}_{D,V}) \ar[r] & 0.
}
\end{equation}
We split the exact sequence \eqref{case-i} into two short exact sequences
$$
\xymatrix@C=0.5cm{
0 \ar[r] & \mathcal{U}_{X} \ar[r] & \CO_{X}^{\oplus 4} \ar[r] & Q \ar[r] & 0 
}
\,
\textrm{ and }
\,
\xymatrix@C=0.5cm{
0 \ar[r] & Q  \ar[r]  & \CO_{Y}(D) \ar[r] & \CH^{0}(\mathcal{E}_{D,V}) \ar[r] & 0,
}
$$
where $Q$ is a torsion sheaf with $\ch(Q)=(0,H,\frac{1}{2}H^{2},\frac{1}{6}H^{3})$.
Then $Q$ is a pure sheaf supported on a hyperplane section $Y\in |H|$.
{Based on the first short exact sequence,
we obtain $\mathrm{Hom}(\CO_{X},Q(-H))= \CN$.
If $f: \CO_{X} \rightarrow Q(-H)$ is a non-trivial morphism, then the image $\mathrm{Im}(f)$ is the structure sheaf of a closed subscheme of $Y$.
Since $Q$ is pure and $Y$ is integral, we obtain $\mathrm{Im}(f)\cong \CO_{Y}\cong Q(-H)$.
Thus, $Q\cong \CO_{Y}(H)$.}

Since $\Hom(\CO_{X},\CO_{Y}(D))=\CN^{4}$ and $\Ext^{i}(\CO_{X},\mathcal{U}_{X})=0$ for all $i$, 
the above two exact sequences implies that $\Hom(\CO_{X},\CH^{0}(\mathcal{E}_{D,V}))=0$.
Since $\chi(\CH^{0}(\mathcal{E}_{D,V})(H))=1$, so we get $\Hom(\CO_{X},\CH^{0}(\mathcal{E}_{D,V})(H))\neq 0$.
Then by the same argument as  Case (2) in the proof of Proposition \ref{may-not-reflexive}, 
we have $\CH^{0}(\mathcal{E}_{D,V})\cong \CO_{\ell}(-H)$ for some line $\ell\subset Y$.
{Since $\ell\subset Y \subset X$, we have $\mathcal{E}xt^{1}(\CO_{\ell},\CO_{Y})\cong \CO_{\ell}(-H)$, $\mathcal{E}xt^{2}(\CO_{\ell},\CO_{Y})\cong \CO_{\ell}$ and $\mathcal{E}xt^{i}(\CO_{\ell},\CO_{Y})=0$ for $i\neq 1,2$.
Hence, we deduce} 
$$
{\Ext^{1}(\CO_{\ell}(-H),\CO_{Y}(H))
\cong H^0\bigl(\ell,\mathcal O_{\ell}(H)\bigr)=\CN^{2}.}
$$
Since $\ell\subset Y$ and $(H+\ell)\cdot \ell=0$, we have the middle row in \eqref{Snake-Lem-exten-OYH-Ol-H}.
Note that any non-zero section
$ s\in H^0\bigl(\ell,\mathcal O_{\ell}(H)\bigr) $
vanishes at exactly one point.
Thus, for any $p\in \ell$, we obtain the last column in \eqref{Snake-Lem-exten-OYH-Ol-H}.
By the Snake lemma for the last two rows in \eqref{Snake-Lem-exten-OYH-Ol-H}, we get the first row in the commutative diagram:
\begin{equation}\label{Snake-Lem-exten-OYH-Ol-H}
\xymatrix@C=0.5cm{
0 \ar[r] & \CO_{Y}(H) \ar[r]\ar[d] & I_{p/Y}(H+\ell) \ar[r]\ar[d] & \CO_{\ell}(-H)\ar[r]\ar[d] & 0\\
0 \ar[r] & \CO_{Y}(H) \ar[r]\ar[d] & \CO_{Y}(H+\ell) \ar[r]\ar[d] & \CO_{\ell}
\ar[r] \ar[d]& 0\\
0\ar[r]& 0 \ar[r] &\CO_{p} \ar[r] & \CO_{p} \ar[r]& 0 }
\end{equation}
Therefore, $\CO_{Y}(D)$ cannot be any extension of $\CO_{\ell}(-H)$ by $\CO_{Y}(H)$, a contradiction.

{\bf Case (ii):} If $s=0$, then $\ch(E_{D,V})=(4,-H, -\frac{5}{6}H^{2},\ch_{3}(\CH^{0}(\mathcal{E}_{D,V}))+\frac{1}{6}H^{3})$. 
We may assume $\ch(\CH^{0}(\mathcal{E}_{D,V}))=(0,0,0, \frac{t}{3}H^{3})$, where $t\geq 0$.
Applying  Proposition  \ref{mainsh-ch3-char} to $E_{D,V}$, 
we get
$$
\ch_{3}(\CH^{0}(\mathcal{E}_{D,V}))+\frac{1}{6}H^{3}=\frac{t}{3}H^{3}+\frac{1}{6}H^{3} \leq \frac{1}{6}H^{3}.
$$
This yields $t=0$. 
Thus, we have $\ch(E_{D,V})=(4,-H,-\frac{5}{6}H^{2}, \frac{1}{6}H^{3})=v$ and $\CH^{0}(\mathcal{E}_{D,V})=0$. 
As a result, $\CO_{Y}(D)$ is globally generated by $V$. 
 
In particular, 
following the proof of \cite[Theorem 6.1]{BBF+}, 
the Bertini's theorem yields that a general section of $\CO_{Y}(D)$ cuts out a smooth curve $C$.
Thanks to the adjunction formula, 
the Chern character of the canonical sheaf $\omega_{C}$ of $C$ is
\begin{eqnarray*}
\ch(\omega_{C}) 
& = & \ch(\CO_{Y}(-H+D)|_{D})= \ch(\CO_{Y}(-H+D))-\ch(\CO_{Y}(-H)) \\
& = &  (0,H,-\frac{1}{6}H^{2},-\frac{1}{2}H^{3})-(0,H,-\frac{3}{2}H^{2},\frac{7}{6}H^{3}) \\
& = &  (0,0,\frac{4}{3}H^{2},-\frac{5}{3}H^{3}).
\end{eqnarray*}
From this, it follows that the smooth curve $C$ is of degree $4$.
Furthermore, the Hirzebruch--Riemann--Roch theorem implies $\chi(\omega_{C})=-1$. 
This means that $C$ is a smooth curve of degree $4$ and genus $0$. 
Since $\RHom(\CO_{X},\CO_{Y}(D))=\CN^{4}$, 
by the definition of left mutation, 
we get 
$$
E_{D}\cong \mathrm{L}_{\CO_{X}}(\CO_{Y}(D))[-1]= \mathrm{L}_{\CO_{X}}(\CO_{Y}(C))[-1].
$$

Finally, since $C$ is degenerate, so there is a short exact sequence
$$
\xymatrix@=0.5cm{
0 \ar[r] & \CO_{X} \ar[r] & I_{C}(H) \ar[r] & I_{C/Y}(H) \ar[r] & 0.
}
$$
Then, applying $\CRHom(-,\CO_{X})[1]$, we get an exact triangle
$$
\xymatrix@=0.5cm{
\CO_{X} \ar[r] & \mathbb{D}(I_{C/Y}(H)) \ar[r] & \mathbb{D}(I_{C}(H)) \ar[r] & \CO_{X}[1].
}
$$
Thus, we derive $\mathrm{L}_{\CO_{X}}(\DD(I_{C}(H))) \cong \mathrm{L}_{\CO_{X}}(\DD(I_{C/Y}(H)))$.
Applying $\CRHom(-,\CO_{X}(-H))[1]$ to the short exact sequence
$$
\xymatrix@=0.5cm{
0 \ar[r] & I_{C/Y} \ar[r] & \CO_{Y} \ar[r] & \CO_{C} \ar[r] & 0,
}
$$
we obtain an exact triangle
$$
\xymatrix@=0.5cm{
\CO_{C}(2)[-1] \ar[r] & \CO_{Y} \ar[r] & \mathbb{D}(I_{C/Y}(H)) \ar[r] & \CO_{C}(2).
}
$$
{
Since $C\subset Y$ and $\Ext^{1}(\CO_{C}(2),\CO_{Y})=\CN$, 
this yields that $\mathrm{L}_{\CO_{X}}(\DD(I_{C}(H)))\cong \mathrm{L}_{\CO_{X}}(\CO_{Y}(C))\cong$}
{$E_{D}[1]$.}
This completes the proof of Proposition \ref{construt-bundle-converse}.
\end{proof}

\begin{rem}
(1) {Let $Y\subset X$ be a hyperplane section and $C\subset Y$ a smooth quartic rational curve.
Then $H^{0}(Y,\CO_{Y}(C))=\CN^{4}$ and $\ch(\CO_{Y}(C))=(0,H,\frac{5}{6}H^{2},-\frac{1}{6}H^{3})$.
By Proposition} \ref{construt-bundle-converse}, {$\CO_{Y}(C)$ is globally generated and $E_{C}$ is a $\mu_{H}$-slope stable vector bundle with Chern character $\ch(E_{C})=(4,-H,-\frac{5}{6}H^{2},\frac{1}{6}H^{3})$.}
(2) Suppose that $C\subset X$ is a smooth twisted cubic. 
Following the proof Proposition \ref{more-vector-bundle}, we can also deduce that
the object $E_{C}:=\mathrm{L}_{\CO_{X}} (\DD(I_{C}(H)))[-1]\in \Ku(X)$ is a $\mu_{H}$-slope stable vector bundle of Chern character $\ch(E_{C})=\ch(I_{\ell})+\ch(S(I_{\ell}))=(3,-H,-\frac{1}{2}H^{2},\frac{1}{6}H^{3})$.
As a matter of fact, this construction gives all stable vector bundles with Chern character $\ch(I_{\ell})+\ch(S(I_{\ell}))$ (see \cite[Theorem 6.1]{BBF+}).
\end{rem}

\begin{rem}\label{new-arugument-for-OYD}
Here is another argument, provided by the referee, showing that $\CO_Y(D)$ cannot be an extension of $ \CO_\ell(-H)$ by $\CO_Y(H)$.
Suppose there is a short exact sequence
\begin{equation}\label{O_YD-contradiction-seq}
\xymatrix@C=0.5cm{
0 \ar[r] & \CO_{Y}(H) \ar[r] & \CO_{Y}(D) \ar[r] & \CO_{\ell}(-H) \ar[r] & 0.
}
\end{equation}
By Grothendieck--Verdier duality, we obtain $\CO_{Y}^{\vee}[1]\cong \CO_{Y}(H)$ and $\CO_{\ell}^{\vee}[2]\cong \CO_{\ell}$.
Taking the derived dual of the exact sequence \eqref{O_YD-contradiction-seq}, we have an exact triangle
$$
\xymatrix@C=0.5cm{
\CO_{Y}(D)^{\vee}[1] \ar[r] & \CO_{Y} \ar[r] & \CO_{\ell}(H) \ar[r] & \CO_{Y}(D)^{\vee}[2].
}
$$
Hence, $\CO_{Y}(D)^{\vee}[1]$ is not a sheaf.
By direct computations, the depth of $\CO_{Y}(D)$ is $2$.
Moreover, since $Y$ is normal and $\CO_{Y}(D)$ is a reflexive sheaf on $Y$, so $\CO_{Y}(D)$ has two-term locally free resolution on $X$.
This implies that $\CO_{Y}(D)^{\vee}[1]$ is a sheaf, a contradiction.
Therefore, $\CO_{Y}(D)$ can not be any extension of $\CO_{\ell}(-H)$ by $\CO_{Y}(H)$.
\end{rem}

We are now ready to finish the proof of Theorem \ref{sh-charact-thm}.

\begin{proof}[Proof of Theorem \ref{sh-charact-thm}]
The ``if part" follows immediately from Proposition \ref{may-not-reflexive} and Proposition \ref{more-vector-bundle}.
Conversely, by Proposition \ref{may-not-reflexive}, 
we may assume that $E$ is a $\mu_{H}$-slope stable reflexive sheaf.
We set $\tilde{E}:=E[1]$.
Then, by \cite[Proposition 4.18]{BBF+}, 
$\tilde{E}$ is $\sigma_{\alpha,\beta}$-stable for $\alpha\gg0$ and $\beta\geq -\frac{1}{4}$.
Since $\ch_{1}(\tilde{E})=H$, so $\tilde{E}$ has no walls along $\beta=0$. 
Thanks to the semicircle $W(\frac{5}{6},\frac{5}{6})$ tangent to the $\alpha$-axis, 
it follows that $W(\frac{5}{6},\frac{5}{6})$ is the greatest wall for $\tilde{E}$.
Note that the numerical wall $W(\tilde{E},\CO_{X}(2H))$ is non-empty and above $W(\frac{5}{6},\frac{5}{6})$.
{By} \cite[{Proposition 2.14}]{BLMS23}, {$\CO_{X}(2H)$ is $\sigma_{\alpha,\beta}$-stable for $\beta\leq  2$.} Since $\tilde{E}$ is tilt stable along $W(\tilde{E},\CO_{X}(2H))$,
by Serre duality, we have
$
\Hom(\tilde{E},\CO_{X}[3]) \cong \Hom(\CO_{X}(2H),\tilde{E})=0.
$
Then, the Hirzebruch--Riemann--Roch theorem yields
$$
\hom(\tilde{E},\CO_{X}[1])=\hom(\tilde{E},\CO_{X})+\hom(\tilde{E},\CO_{X}[2])-\chi(\tilde{E},\CO_{X}) \geq -\chi(\tilde{E},\CO_{X})=3.
$$
Along $W(\frac{5}{6},\frac{5}{6})$, $\tilde{E}$ and $\CO_{X}[1]$ have the same tilt-slope.
Since $\CO_{X}[1]$ is tilt stable,
any non-trivial morphism $\tilde{E}\rightarrow \CO_{X}[1]$ is surjective and destabilizes $\tilde{E}$ below $W(\frac{5}{6},\frac{5}{6})$. 
Namely, there are no tilt semistable objects below $W(\frac{5}{6},\frac{5}{6})$ with Chern character $-v$.

Now we let $r:=\hom(\tilde{E},\CO_{X}[1])\geq 3$.
Then we have the following short exact sequence of tilt semistable objects along $W(\frac{5}{6},\frac{5}{6})$
\begin{equation}\label{Construct-main-sh-SSE}
\xymatrix@C=0.5cm{
0 \ar[r] & G \ar[r] & \tilde{E} \ar[r] &  \CO_{X}^{\oplus r}[1] \ar[r] & 0.
}
\end{equation}
If $r\geq 5$, then we have
$
(\frac{5}{6})^{2}> \frac{\frac{23}{3}}{4r(r-4)}.
$
This is a contradiction with \cite[Proposition 4.16]{BBF+}.
Hence, we get $r=3$ or $r=4$.
In the following, we discuss these two cases in detail.
 
{\bf Case 1.} 
If $r=4$, then $\ch(G)=(0,H,\frac{5}{6}H^{2},-\frac{1}{6}H^{3})$. 
Since $E\in \Ku(X)$, applying $\Hom(\CO_{X},-)$ to \eqref{Construct-main-sh-SSE}, 
we get $\Hom(\CO_{X},G)=\CN^{4}$.
Suppose that $G$ is not of the form $\CO_{Y}(D)$ for a Weil divisor $D$ on $Y\in |H|$.
Since $G$ is tilt semistable along $W(\frac{5}{6},\frac{5}{6})$, 
by Lemma \ref{unique-wall-for-tor-sh} and Theorem \ref{Classification-torison-sh}, 
we have $\CH^{-1}(G)=\CO_{X}$.
Since $\Hom(\CO_{X},E)=0$, this contradicts the exact sequence \eqref{Construct-main-sh-SSE}. 
As a consequence, from Proposition \ref{construt-bundle-converse}, it follows that
there is a degenerate smooth quartic rational curve $C$ of class $D$ such that $E\cong \mathrm{L}_{\CO_{X}}(\DD(I_{C}(H)))[-1]$.

{\bf Case 2.} 
If $r=3$, then $\Hom(E,\CO_{X}[1])=0$ and $\ch(G)=(-1,H,\frac{5}{6}H^{2},-\frac{1}{6}H^{3})$. 
Since $G$ has no walls along $\beta=0$, thus $G$ is $\sigma_{\alpha,0}$-semistable for $\alpha>0$.
By \cite[Proposition 5.1.3]{BMT14}, there exists an exact triangle 
\begin{equation}\label{dual-ss-stable-thm1.2}
\xymatrix@C=0.5cm{
\tilde{G} \ar[r] & \DD(G) \ar[r] &T[-1] \ar[r] &  \tilde{G}[1],
}
\end{equation}
where $\tilde{G} $ is $\sigma_{\alpha,0}$-semistable and $T$ is a torsion sheaf supported in dimension zero.
This yields $\CH^{-1}(\DD(G))\cong \CH^{-1}(\tilde{G})$, $\CH^{0}(\DD(G))\cong \CH^{0}(\tilde{G})$, $\CH^{1}(\DD(G))\cong T$, and $\CH^{i}(\DD(G))=0$ for $i\neq -1,0,1$.
Since $\ch_{1}(\tilde{G})=\ch_{1}(\DD(G))=H$, 
so $\tilde{G}$ has no wall along $\beta=0$.
Thanks to Proposition \ref{limit-tilt-stab-Gie-stab-prop}, $\tilde{G}$ is a $\mu_{H}$-slope stable sheaf.

Note that $\mu_{H}(G)=-1<0$.
According to \cite[Proposition 4.9]{BBF+},
$\CH^{-1}(G)$ is a $\mu_{H}$-slope semistable reflexive sheaf and $\ch_{1}(\CH^{-1}(G))=-H$.
Hence, we get $\CH^{-1}(G)\cong \CO_{X}(-H)$ and thus $\ch(\CH^{0}(G))=(0,0,\frac{4}{3}H^{2},-\frac{1}{3}H^{3})$.
Then there exists a short exact sequence
\begin{equation}\label{case2-thm1.2-ses1}
\xymatrix@C=0.5cm{
0 \ar[r] & \CO_{X}(-H)[1] \ar[r]^{} & G \ar[r] & \CH^{0}(G) \ar[r] & 0.
}
\end{equation}
Applying $\CRHom(-,\CO_{X})$ to \eqref{case2-thm1.2-ses1}, we get 
$
\CExt^{3}(\CH^{0}(G),\CO_{X})\cong \CH^{2}(\DD(G))=0.
$
 
Set $\ch_{3}(T):=\frac{k}{3}H^{3}$, where $k\geq 0$.
Then we have $\ch(\tilde{G})=(1,H,-\frac{5}{6}H^{2},(\frac{k}{3}-\frac{1}{6})H^{3})$.
Since $\hom(E,\CO_{X})=3$, the short exact sequence \eqref{Construct-main-sh-SSE} implies $\Hom(G,\CO_{X}[1])=0$.
Applying $\Hom(\CO_{X},-)$ to \eqref{dual-ss-stable-thm1.2}, we have
$$
\Hom(\CO_{X},\tilde{G})\cong \Hom(\CO_{X},\DD(G))\cong \Hom(G,\CO_{X}[1])=0.
$$
Since $\mu_{H}(\tilde{G})=1>-2$, by $\mu_{H}$-slope stability and Serre duality, we have
$$
\Ext^{3}(\CO_{X},\tilde{G})\cong \Hom(\tilde{G},\CO_{X}(-2H))=0.
$$
Note that both $\tilde{G}$ and $\CO_{X}(-2H)[1]$ are $\sigma_{\alpha,0}$-stable.
Since the tilt-slopes 
$$
\lim_{\alpha\to 0} \mu_{\alpha,0}(\tilde{G})=-\frac{5}{3}>-2 =\lim_{\alpha\to 0} \mu_{\alpha,0}(\CO_{X}(-2H)[1]),
$$
by Serre duality, we get 
$$
\Ext^{2}(\CO_{X},\tilde{G})\cong \Hom(\tilde{G},\CO_{X}(-2H)[1])=0.
$$
Then, by the Hirzebruch--Riemann--Roch theorem, we have
$$
\chi(\tilde{G})=k=-\hom(\CO_{X},\tilde{G}[1])\leq 0
$$
and thus $k=0$.

As a result, we get $T=0$ and then $\tilde{G}\cong \DD(G)$.
Hence, the Chern character $\ch(\tilde{G})=(1,H,-\frac{5}{6}H^{2},-\frac{1}{6}H^{3})$.
Twisted by $\CO_{X}(-H)$, 
we have $\ch(\tilde{G}(-H))=(1,0,-\frac{4}{3}H^{2},H^{3})$.
Since $\tilde{G}$ is a $\mu_{H}$-slope stable sheaf, 
$\tilde{G}(-H)\cong I_{C}$ is the ideal sheaf of a $1$-dimensional closed subscheme $C\subset X$ and $\ch(\CO_{C})=(0,0,\frac{4}{3}H^{2},-H^{3})$.
As a consequence, we obtain $\DD(G)\cong I_{C}(H)$ and thus $G\cong \DD(I_{C}(H))$.
Then we derive $\CH^{0}(G) \cong  \CExt^{1}(I_{C}(H),\CO_{X})$ and $\CHom(I_{C}(H),\CO_{X})\cong \CO_{X}(-H)$.
Applying $\CHom(-,\CO_{X})$ to the short exact sequence
\begin{equation}\label{pf-1.2-C-curve-eseq}
\xymatrix@C=0.5cm{
0 \ar[r] & I_{C}(H) \ar[r]^{} & \CO_{X}(H) \ar[r] & \CO_{C}(H) \ar[r] & 0,
}
\end{equation}
we obtain $\CExt^{i}(\CO_{C}(H),\CO_{X}) =0$ for $i\neq 2$ and 
$$
\CExt^{2}(\CO_{C}(H),\CO_{X})\cong  \CExt^{1}(I_{C}(H),\CO_{X})\cong \CH^{0}(G).
$$
{Thus, $\CO_{C}$ is a pure sheaf.
It follows that $C$ is a locally Cohen–Macaulay curve of degree $4$ and arithmetic genus $p_{a}(C)=0$. 
Since $E\in \Ku(X)$, applying $\Hom(-,\CO_{X}(-H))$ to \eqref{Construct-main-sh-SSE} and by Serre duality, we derive 
$$
\Ext^{2}(\mathbb{D}(I_{C}(H)),\CO_{X}(-H))\cong \Ext^{2}(E[1],\CO_{X}(-H))\cong \Ext^{2}(\CO_{X}(H),E)=0.
$$
Then $\Ext^{1}(\CO_{X}(H),I_{C}(H))=0$.
Applying $\Hom(\CO_{X}(H),-)$ to \eqref{pf-1.2-C-curve-eseq}, we have $$\Hom(\CO_{X}(H),\CO_{C}(H))\cong 
\Hom(\CO_{X}(H),\CO_{X}(H))=\CN.     
$$
Thus, $H^{0}(C,\CO_{C})=\CN$.

Finally, 
applying $\CRHom(-,\CO_{X})$ to} \eqref{Construct-main-sh-SSE}, 
{we get $\CExt^{i}(E,\CO_{X})=0$ for $i= 1,2,3$ and a short exact sequence}
\begin{equation}\label{more-vb-ses-thm1.2}
{\xymatrix@C=0.5cm{
0 \ar[r] & \CO_{X}^{\oplus 3} \ar[r]^{} & E^{\vee} \ar[r] &  I_{C}(H)  \ar[r] & 0.}}
\end{equation}
{Hence, $E$ is a vector bundle. 
Applying $\Hom(\CO_{X},-)$ and $\Hom(\CO_{\ell}(H), -)$ to} \eqref{more-vb-ses-thm1.2} {respectively,
we obtain $\Hom(\CO_{X},I_{C}(H))=0$ and $\Hom(\CO_{\ell}, \CO_{C})\cong \Ext^{1}(\CO_{\ell}, I_{C})=0$, where $\ell \subset X$ is a line.
This implies that $C$ is a non-degenerate and line-free.
Consequently, by the short exact sequence} \eqref{Construct-main-sh-SSE}, 
{we obtain an exact triangle}
$$
\xymatrix@C=0.5cm{
\CO_{X}^{\oplus 3}  \ar[r] & \DD(I_{C}(H)) \ar[r] & E[1] \ar[r] &  \CO_{X}^{\oplus 3}[1].
}
$$
{Namely, we have $E=\mathrm{L}_{\CO_{X}}(\DD(I_{C}(H)))[-1]$, where $C$ is a line-free non-degenerate a locally Cohen–Macaulay curve of degree $4$,  $p_{a}(C)=0$ and $H^{0}(C,\CO_{C})=\CN$.
}
This completes the proof of Theorem \ref{sh-charact-thm}.
\end{proof}


\end{document}